\documentclass{amsart}
\title{Pullback Attractors for Generalized Evolutionary Systems}
\author{Alexey Cheskidov}
    \address{Department of Mathematics, Statistics, and Computer Science\\
        University of Illinois at Chicago\\
        322 Science and Engineering Offices (M/C 249)\\
        851 S. Morgan Street\\
        Chicago, Illinois 60607-7045 USA}
    \email{acheskid@uic.edu}
\author{Landon Kavlie}
    \address{Department of Mathematics, Statistics, and Computer Science\\
        University of Illinois at Chicago\\
        322 Science and Engineering Offices (M/C 249)\\
        851 S. Morgan Street\\
        Chicago, Illinois 60607-7045 USA}
    \email{lkavli2@uic.edu}
\thanks{The authors were partially supported by NSF Grant DMS-1108864.}
\usepackage{amsthm}
\usepackage{amsfonts}
\usepackage{amsmath}
\usepackage{mathrsfs}
\usepackage{graphicx}
\usepackage{enumitem}
\usepackage{pgfplots}
\usepackage{hyperref}

\newcommand{\ddr}{\mathrm{dr}}
\newcommand{\dds}{\mathrm{ds}}

\newcommand{\ddx}{\mathrm{dx}}

\newcommand{\dxi}{\mathrm{d}\xi}
\newcommand{\derivt}{\mathrm{\frac{d}{dt}}}

\newcommand{\supp}{\mathrm{supp}}

\newcommand{\ee}{\epsilon}
\newcommand{\ww}{\omega}
\newcommand{\WW}{\Omega}
\newcommand{\Sig}{\Sigma}
\newcommand{\sig}{\sigma}

\newcommand{\ds}{\mathrm{d_s}}
\newcommand{\dw}{\mathrm{d_w}}
\newcommand{\dd}{\mathrm{d_\bullet}}
\newcommand{\wwd}{\omega_\bullet}
\newcommand{\wws}{\omega_\mathrm{s}}
\newcommand{\www}{\omega_\mathrm{w}}
\newcommand{\WWd}{\Omega_\bullet}
\newcommand{\WWs}{\Omega_\mathrm{s}}
\newcommand{\WWw}{\Omega_\mathrm{w}}
\newcommand{\N}{\mathbb{N}}
\newcommand{\Z}{\mathbb{Z}}
\newcommand{\R}{\mathbb{R}}

\newcommand{\AAs}{\mathscr{A}_\mathrm{s}}
\newcommand{\AAw}{\mathscr{A}_\mathrm{w}}
\newcommand{\AAd}{\mathscr{A}_\bullet}
\newcommand{\Xs}{X_\mathrm{s}}
\newcommand{\Xw}{X_\mathrm{w}}

\newcommand{\BB}{\mathscr{B}}
\newcommand{\EE}{\mathscr{E}}
\newcommand{\FF}{\mathscr{F}}
\newcommand{\II}{\mathscr{I}}
\newcommand{\PP}{\mathscr{P}}
\newcommand{\TT}{\mathscr{T}}

\newcommand{\abs}[1]{\lvert#1\rvert}
\newcommand{\norm}[1]{\lVert#1\rVert}
\newcommand{\ip}[2]{\left(#1,#2\right)}
\newcommand{\IP}[2]{\left(\left(#1,#2\right)\right)}
\newcommand{\angip}[2]{\langle#1,#2\rangle}

\newcounter{icount}
\newcommand{\refitem}[1]{\item[#1]\refstepcounter{icount}\label{#1}}
\newcommand{\iref}[1]{\hyperref[#1]{#1}}

\pgfplotsset{soldot/.style={color=black,only marks,mark=*}} 
\pgfplotsset{holdot/.style={color=black,fill=white,only marks,mark=*}}

\begin{document}

\newtheorem*{thm*}{Theorem}
\newtheorem{thm}{Theorem}[section]
\newtheorem{prop}[thm]{Proposition}
\newtheorem{lem}[thm]{Lemma}
\newtheorem{cor}[thm]{Corollary}
\newtheorem{define}[thm]{Definition}
\newtheorem{rmk}[thm]{Remark}

\begin{abstract}
We give an abstract framework for studying nonautonomous PDEs, called a 
generalized evolutionary system. In this setting, we define the notion of a
pullback attractor. Moreover, we show that the pullback attractor, in the weak
sense, must always exist. We then study the structure of these attractors and
the existence of a strong pullback attractor. We then apply our framework to
both autonomous and nonautonomous evolutionary systems as they first appeared
in earlier works by Cheskidov, Foias, and Lu. In this context, we compare the
pullback attractor to both the global attractor (in the autonomous case) and
the uniform attractor (in the nonautonomous case). Finally, we apply our
results to the nonautonomous 3D Navier-Stokes equations on a periodic domain
with a translationally bounded force. We show that the Leray-Hopf weak
solutions form a generalized evolutionary system and must then have a weak
pullback attractor.
\end{abstract}


\maketitle

\section{Introduction}
Uniform attractors have been extensively studied since their introduction by 
Haraux as the minimal compact set which attracts all the trajectories starting 
from a bounded set uniformly with respect to the initial time \cite{H91}. One 
approach that has been widely accepted uses the tools developed by Chepyzhov 
and Vishik \cite{CV94}, \cite{CV02}. Their method involves the use of a time 
symbol and a family of processes. The uniform attractor, under sufficient 
conditions is fibered over the symbol space into ``kernel sections." These 
kernel sections, however, do not have classical attraction properties. The 
attraction is in a pullback sense, letting the initial time go to minus 
infinity \cite{CV93}. The concept of a pullback attractor originated in the 
work of Crauel, Flandoli, Kloeden, and Schmalfuss \cite{CF94}, \cite{KS97}. 
For more information on pullback attractors, see the books of Kloeden and 
Rasmussen \cite{KR11}, Carvalho et al \cite{CLR13}, and Cheban \cite{C04}. 

The question of uniqueness for the 3D Navier-Stokes equations (NSEs) is still 
unresolved. Even so, many frameworks exist for studying the asymptotic 
dynamics of a system evolving according to the 3D NSEs, without assuming 
uniqueness of the solutions. For a comparison between two canonical 
frameworks, see Caraballo et al \cite{CMR03}. In their paper, they compare 
the framework of multivalued semiflows used by Melnik and Valero \cite{MV98} 
to the framework of generalized semiflows developed by Ball \cite{B97}. The 
first approach involves a set-valued function with some inclusion properties 
that account for the lack of uniqueness. The approach used by Ball involves 
defining trajectories in the phase space, keeping in mind that there may be 
more than one trajectory starting from a single starting value. Ball's 
construction requires certain assumptions about trajectories including the 
ability to concatenate them which are still not known for the Leray-Hopf weak 
solutions to the 3D NSEs. 

In their paper \cite{CF06}, Cheskidov and Foias introduced the concept of an 
evolutionary system allowing them to construct a framework based on known 
results for the Leray-Hopf weak solutions to the autonomous 3D NSEs. Moreover, 
they defined the concept of a global attractor in this setting. Cheskidov 
later added the idea of a trajectory attractor for an evolutionary system 
\cite{C09}. In order to deal with nonautonomous systems, Cheskidov and Lu 
introduced the concept of a nonautonomous evolutionary system and have applied 
it to the 3D NSEs and certain reaction-diffusion equations (\cite{CL09}, 
\cite{CL12}). In this paper, we introduce the idea of a generalized 
evolutionary  system, a generalization of the previous concepts where we 
remove the ability to ``shift trajectories." In this setting, we explore the 
existence of pullback attractors as well as their relationship to the 3D NSEs. 
We also explore the relationship between pullback attractors for generalized 
evolutionary systems, global attractors for autonomous evolutionary systems, 
and uniform attractors for nonautonomous evolutionary systems.

As in the autonomous case, several frameworks exist for studying the 
nonautonomous dynamical systems without uniqueness. We describe two of the 
those theories. The first is the theory of trajectory attractors 
(\cite{S96}, \cite{FS99}, \cite{CV02}, \cite{SY02}). This approach studies 
trajectories as points in the ``space of trajectories" (such as $C([0,\infty);
L^2)$). Then, one studies the attraction properties in this space of 
trajectories. One can then try to project down to the phase space by using an 
evaluation map and study the attraction properties of these projections. The 
second technique is the use of multivalued processes (\cite{KV07}, 
\cite{V11}). 

Although our framework follows trajectories which exist in a ``space of 
trajectories," our framework differs from the framework of trajectory 
attractors. We use trajectories only to follow the evolution of points in the 
phase space. Thus, we develop attraction properties only in the phase space. 
Our framework also differs from the classical framework of multivalued 
processes in that we follow individual trajectories taking a single point to 
another single point in the phase space. If one were to union together all of 
the possible ending points from a given starting point, one could define a 
multivalued process from individual trajectories. Also different from either 
of these frameworks, we exploit, simultaneously, both the strong and weak 
topologies on our space. To our knowledge, this gives the first proof of the 
existence of a weak pullback attractor for the Leray-Hopf weak solutions of 
the 3D NSEs using the known properties of the Leray-Hopf weak solutions. The 
outline of our paper is given below.

First, in Section~\ref{setup}, we define a generalized evolutionary system 
$\EE$, a pullback attractor $\AAd(t)$, and the concept of a pullback 
omega-limit set $\WWd(A,t)$. Here, $\bullet$ represents either $w$ or $s$ to 
represent the fact that we are considering our phase space $X$ with two 
topologies, the weak and the strong topology, respectively. We then show 
the existence of the weak pullback attractor $\AAw(t)$ and give a 
characterization of the strong pullback attractor $\AAs(t)$ (if it exists) in 
terms of pullback omega-limits. Moreover, we relate the concept of a 
generalized evolutionary system to the classical theory of processes as given 
in \cite{CLR13}. We follow this with a short section of worked examples in 
Section~\ref{examples}. Here, we apply our framework to both abstract and 
physical situations to illustrate the possible relationships between the weak 
and strong pullback attractors.

In Section~\ref{strongPAC}, we introduce the concept of pullback asymptotic 
compactness. We then show that in this setting, the weak pullback attractor 
$\AAw(t)$ is, in fact, a strongly compact strong pullback attractor $\AAs(t)$. 
This generalizes classical results using asymptotic compactness in the 
autonomous setting. Next, in Section~\ref{tracking}, we add the assumption 
that each family of trajectories is compact in $C([s,\infty);\Xw)$. This is 
true of the Leray-Hopf solutions for the nonautonomous 3D NSEs under certain 
assumptions. We introduce the ideas of pullback invariance, pullback 
semi-invariance, and pullback quasi-invariance. We then show that the family 
of weak omega-limit sets $\WWw(A,t)$ is pullback quasi-invariant. This allows 
us to give our first characterization of the weak pullback attractor in terms 
of complete trajectories. That is, the weak pullback attractor is the maximal 
pullback invariant and maximal pullback quasi-invariant subset of the phase 
space. Moreover, we have a weak pullback tracking property 
(Theorem~\ref{weak=i}). If $\EE$ satisfies the property of being pullback 
asymptotically compact, then we have that the strong pullback attractor 
$\AAs(t)$ is the maximal pullback invariant and maximal pullback 
quasi-invariant set in the phase space. Moreover, we have a characterization 
of $\AAs(t)$ in terms of complete trajectories as well as a strong pullback 
tracking property (Theorem~\ref{strong=i}).

For Section~\ref{energy}, we add other assumptions to our generalized 
evolutionary system which are known for Leray-Hopf weak solutions to the 
nonautonomous 3D NSEs with an appropriate forcing term. Namely, an energy 
inequality and strong convergence a.e. for a weakly convergent sequence of 
trajectories. Under these assumptions, along with the compactness in 
$C([s,\infty);\Xw)$ from Section~\ref{tracking}, and the assumption that 
complete trajectories are strongly continuous (that is, $\EE((-\infty,\infty))
\subseteq C((-\infty,\infty);\Xs)$), we can show that the generalized 
evolutionary system is actually pullback asymptotically compact. This is a 
generalization of the corresponding result in \cite{C09}. Earlier results of 
this kind can be found in \cite{B97} and \cite{R06}.

The next section, Section~\ref{relate} relates results already discovered for 
evolutionary systems given in \cite{CF06}, \cite{C09}, and \cite{CL12} to our 
new generalized framework. We begin by recalling the basic definitions for an 
autonomous evolutionary system (as given in \cite{CF06} and \cite{C09}) and 
a nonautonomous evolutionary system (as given in \cite{CL09} and \cite{CL12}). 
In the autonomous case, we prove that the global attractor $\AAd$ for an 
autonomous evolutionary system exists if and only if the pullback attractor 
$\AAd(t)$ exists. Moreover, we have that for each $t$, $\AAd=\AAd(t)$. In the 
nonautonomous case, we show that the $\dd$-uniform attractor (if it exists) 
$\AAd^\Sig$ always contains the union of all the pullback Omega-limits. That 
is, in Theorem~\ref{unifsections}, we show that
\begin{equation*}
\overline{\bigcup_{\sig\in\Sig}\WWd^\sig(X,t_0)}^\bullet\subseteq\AAd^\Sig
\end{equation*} 
where $\WWd^\sig(X,t_0)$ is the $\dd$-pullback omega limit set at an arbitrary 
fixed time $t_0$ for the generalized evolutionary system $\EE_\sig$ with the 
fixed symbol $\sig\in\Sig$. As before, using the existence of the weak uniform 
and weak pullback attractors, Corollary~\ref{generalweakunifpullback} gives us 
that 
\begin{equation*}
\overline{\bigcup_{\sig\in\Sig}\AAw^\sig(t_0)}^w\subseteq\AAw^\Sig.
\end{equation*}
On the other hand, with additional assumptions on the nonautonomous 
evolutionary system $\EE_\Sig$ and each generalized evolutionary system 
$\EE_\sig$ for $\sig\in\Sig$ fixed, we use known characterizations of the 
weak uniform attractor $\AAw^\Sig$ and the weak pullback attractor 
$\AAw^\sig(t)$ to say in Theorem~\ref{unif=pullbacksections} that 
\begin{equation*}
\AAw^\sig=\overline{\bigcup_{\sig\in\Sig}\AAw^\sig(0)}^w.
\end{equation*}
Finally, we add additional assumptions to ensure asymptotic compactness of 
$\EE_\Sig$ and each $\EE_\sig$. This then guarantees the existence of a 
strongly compact strong uniform attractor for $\EE_\Sig$ and a strongly 
compact strong pullback attractor for each $\EE_\sig$. In this case, as shown 
in Theorem~\ref{strongsections},
\begin{equation*}
\AAs^\Sig = \AAw^\Sig = \overline{\bigcup_{\sig\in\Sig}\AAw^\sig(0)}^w 
                      = \overline{\bigcup_{\sig\in\Sig}\AAs^\sig(0)}^w.
\end{equation*}
This harkens back to the classical theory given in \cite{CV02} which shows 
that the uniform attractor, under certain assumptions, is fibered. Moreover, 
the sections of the uniform attractor have attraction properties which are in 
a pullback sense. For a discussion on the relationship between pullback and 
uniform attractors using the framework of multivalued processes, we refer the 
reader to \cite{CLMV03}.

Finally, Section~\ref{3dnses} closes with an application of our setup to 
Leray-Hopf weak solutions of the nonautonomous 3D NSEs. We use a periodic 
setup along with a translationally bounded force. In this setting, we prove 
that Leray solutions (which are continuous at the starting time) converge in a 
to an absorbing ball in $L^2$ which is weakly compact. Using this, we show 
that the Lery-Hopf weak solutions form a generalized evolutionary system. 
Therefore, by the previously-developed theory, there exists a weak pullback 
attractor. In fact, we have that 
\begin{equation*}
\AAw(t)=\{u(t):u~\mathrm{is~a~complete~bounded~solution}\}.
\end{equation*}
This generalizes the results Foias and Temam \cite{FT85}. Moreover, if the 
force is assumed to be normal, and if we add the assumption that complete 
trajectories, Leray-Hopf solutions which exist for all $t\in(-\infty,\infty)$, 
are strongly continuous, in $C((-\infty,\infty);\Xs)$, then the three 
properties listed above apply and the given generalized evolutionary system is 
pullback asymptotically compact. In this case, the weak attractor is a 
strongly compact, strong pullback attractor.

\section{Generalized Evolutionary System}
\label{setup}
\subsection{Preliminaries}

We start with the setup as it first appeared in \cite{CF06}. So, let 
$(X,\ds(\cdot,\cdot))$ be a metric space with a metric $\ds$ known as the 
strong metric on $X$. Let $\dw$ be another metric on $X$ satisfying the 
following conditions:
\begin{enumerate}
\item $X$ is $\dw$ compact. 
\item If $\ds(u_n,v_n)\rightarrow0$ as $n\rightarrow\infty$ for some $u_n,v_n
\in X$ then $\dw(u_n,v_n)\rightarrow0$ as $n\rightarrow\infty$.
\end{enumerate}
As justified by property (2), we will call $\dw$ the weak metric on $X$. Denote 
by $\overline{A}^\bullet$ the closure of the set $A\subseteq X$ in the topology 
generated by $\dd$. Note that any strongly compact set ($\ds$-compact) is also 
weakly compact ($\dw$-compact), and any weakly closed set ($\dw$-closed) is 
also strongly closed ($\ds$-closed).

Let $C([a,b];X_\bullet)$, where $\bullet=$ s or w, be the space of $\dd$-
continuous $X$-valued functions on $[a,b]$ endowed with the metric
\begin{equation*}
\mathrm{d}_{C([a,b];X_\bullet)}(u,v):=\sup_{t\in[a,b]}\dd(u(t),v(t)).
\end{equation*}
Let also $C([a,\infty);X_\bullet)$ be the space of all $\dd$-continuous 
$X$-valued functions on $[a,\infty)$ endowed with the metric
\begin{equation*}
\mathrm{d}_{C([a,\infty);X_\bullet)}(u,v):=\sum_{n\in\N}\frac{1}{2^n}
 \frac{\sup\{\dd(u(t),v(t)):a\le t\le a+n\}}{1+\sup\{\dd(u(t),v(t)):a\le t
    \le a+n\}}.
\end{equation*}

Let 
\begin{equation*}
\TT:=
  \{I\subset\R:I=[T,\infty)\mathrm{~for~some~}T\in\R\}\cup\{(-\infty,\infty)\},
\end{equation*}
and for each $I\in\TT$, let $\FF(I)$ denote the set of all $X$-valued functions 
on $I$.

\begin{define}
\label{GES}
A map $\EE$ that associates to each $I\in\TT$ a subset $\EE(I)\subset\FF(I)$ 
will be called a generalized evolutionary system if the following conditions 
are satisfied:
\begin{enumerate}
\item $\EE([\tau,\infty))\ne\emptyset$ for each $\tau\in\R$.
\item $\{u(\cdot)|_{I_2}:u(\cdot)\in\EE(I_1)\}\subseteq\EE(I_2)$ for each $I_1, 
I_2\in\TT$ with $I_2\subseteq I_1$.
\item $\EE((-\infty,\infty))=\{u(\cdot):u(\cdot)|_{[T,\infty)}\in\EE([T,\infty))
~\forall T\in\R\}$.
\end{enumerate}
\end{define}
We will refer to $\EE(I)$ as the set of all trajectories on the time 
interval $I$. Trajectories in $\EE((-\infty,\infty))$ are called complete. 
Next, for each $t\ge s\in\R$ and $A\subseteq X$, we define the map 
\begin{equation*}
P(t,s):\PP(X)\rightarrow\PP(X),
\end{equation*}
\begin{equation*}
P(t,s)A:=\{u(t):u(s)\in A, u\in\EE([s,\infty))\}.
\end{equation*}
We get, for each $t\ge s\ge r\in\R$ and $A\subseteq X$
\begin{equation*}
P(t,r)A\subset P(t,s)P(s,r)A.
\end{equation*}

We will also study generalized evolutionary systems endowed with the 
following properties:
\begin{enumerate}
\refitem{A1}$\EE([s,\infty))$ is compact in $C([s,\infty);\Xw)$ for 
each $s\in\R$.
\refitem{A2} (Energy Inequality) Let $X$ be a set in some Banach space 
$H$ satisfying the Radon-Riesz Property (see below) with norm $\abs{\cdot}$ so 
that $\ds(x,y)=\abs{x-y}$ for each $x,y\in X$, and assume that $\dw$ induces 
the weak topology on $X$. Assume that for each $\ee>0$ and each $s\in\R$ there 
is a $\delta:=\delta(\ee,s)$ so that for every $u\in\EE([s,\infty))$ and 
$t>s\in\R$ 
\begin{equation*}
\abs{u(t)}\le\abs{u(t_0)}+\ee
\end{equation*}
for $t_0$ a.e. in $(t-\delta,t)$.
\refitem{A3} (Strong Convergence a.e.) Let $u,u_n\in\EE([s,\infty))$ be 
so that $u_n\rightarrow u$ in $C([s,t];\Xw)$ for some $s\le t\in\R$. Then, 
$u_n(t_0)\xrightarrow{\ds}u_n(t_0)$ for a.e. $t_0\in[s,t]$.
\end{enumerate}

\begin{rmk}
A Banach space $H$ with norm $\abs{\cdot}$ satisfies the Radon-Riesz property 
if $x_n\rightarrow x$ in norm if and only if $x_n\rightarrow x$ weakly and 
\begin{equation*}
\lim_{n\rightarrow\infty}\abs{x_n}=\abs{x}.
\end{equation*}
Often, $X$ will be a closed, bounded subset of a separable, reflexive Banach 
space. By the Troyanski Renorming Theorem, we can assume that our norm makes 
$H$ a locally uniformly convex space, at which point the Radon-Riesz property 
is satisfied.
\end{rmk}

To see how this relates back to the classical setting, let $H$ be a separable, 
reflexive Banach space, which we call the phase space. Let $S(\cdot,\cdot)$ be 
a process on $H$. That is, for each $t\ge s$, we have that $S(t,s):H
\rightarrow H$ with the following properties:
\begin{align*}
S(t,s) &= S(t,r)S(r,s) \\
S(t,t) &= Id_H
\end{align*}
for any $t\ge r\ge s$. A trajectory on $H$ is a mapping $u:[s,\infty)
\rightarrow H$ so that $u(t)=S(t,s)u(s)$ for each $t\ge s$. A set $X\subseteq 
H$ will be called absorbing if, for each $s\in\R$ and $B\subseteq H$ bounded, 
there is $t_0:=t_0(B,s)$ so that for $t\ge t_0$, 
\begin{equation*}
S(t,s)B\subseteq X.
\end{equation*}
If there exists a closed absorbing ball $X$, then we call the process $S$ 
dissipative. 

If $S$ is dissipative, and we can ensure that it is dissipative arbitrarily 
far in the past (that is, for each $s\in\R$, there is a trajectory 
$u:[s,\infty)\rightarrow X$), then studying the asymptotic pullback dynamics 
of $S$ on $H$ amounts to studying the asymptotic pullback dynamics of $S$ on 
$X$. That is, using the definition of the pullback attractor 
Definition~\ref{attractor}, one can show that if $X$ has a pullback attractor 
$\mathscr{A}(t)$ (by restricting $S$ to $X$), then $\mathscr{A}(t)$ is a 
pullback attractor for $H$. Note that since $H$ is a separable reflexive 
Banach space, both the strong and weak topologies on $X$ are metrizable. We 
define a generalized evolutionary system on $X$ by 
\begin{equation*}
\EE([s,\infty)):=\{u(\cdot):u(t)=S(t,s)u(s), u(t)\in X~\forall t\ge s\}.
\end{equation*}
In particular, this also gives us the following characterization for each 
$t\gg s\in\R$ and $A\subseteq X$
\begin{equation*}
P(t,s)A=S(t,s)A.
\end{equation*}

As we will see later, by Theorem~\ref{weakattracts} and 
Theorem~\ref{weakattractorexists} that the weak pullback attractor exists for 
$\EE$ and 
\begin{equation*}
\AAw(t)=\WWw(X,t)=\bigcap_{s\le t}\overline{\bigcup_{r\le s}S(t,r)X}^w.
\end{equation*}
Moreover, if we know that \iref{A1} holds (that $\EE([s,\infty))$ is compact 
in $C([s,\infty);\Xw)$ for each $s\in\R$), then we get, using 
Theorem~\ref{weak=i} that 
\begin{equation*}
\AAw(t)=\{u(t):u\in\EE((-\infty,\infty))\}.
\end{equation*}
Finally, if we also have that \iref{A2} and \iref{A3} also hold and complete 
trajectories are strongly continuous $\EE((-\infty,\infty))\subseteq 
C((-\infty,\infty);X)$, then by Corollary~\ref{123implystrong}, $\EE$ 
possesses a strongly compact, strong pullback attractor $\AAs(t)$. In fact, 
by Corollary~\ref{a1weak=strong}, $\AAs(t)=\AAw(t)=\{u(t):
u\in\EE((-\infty,\infty))\}$.

\subsection{Pullback Attracting Sets, $\WW$-limits, and Pullback Attractors}

Let $\EE$ be a fixed generalized  evolutionary system on a metric space $X$. 
For $A\subseteq X$ and $r>0$, denote $B_\bullet(A,r):=\{x\in X:\dd(x,A)<r\}$,  
where 
\begin{equation*}
\dd(x,A):=\inf_{a\in A}\dd(x,a),~\bullet=\mathrm{s,w}.
\end{equation*}
A family of sets $A(t)\subseteq X$, $t\in\R$ (uniformly) pullback attracts a 
set $B\subseteq X$ in the $\dd$-metric ($\bullet=$s, w) if for any $\ee>0$, 
there exists an $s_0:=s_0(B,\ee,t)<t\in\R$ so that for $s\le s_0$, 
\begin{equation*}
P(t,s)B\subseteq B_\bullet(A(t),\ee).
\end{equation*}

\begin{define}
A family of sets $A(t)\subseteq X$ for $t\in\R$ are $\dd$-pullback attracting 
($\bullet=$s, w) if they pullback attract $X$ in the $\dd$-metric.
\end{define}

\begin{define}
\label{attractor}
A family of sets $\AAd(t)\subseteq X$ is the $\dd$-pullback attractor of $X$ 
if for each $t$, $\AAd(t)$ is $\dd$-closed, $\dd$-pullback attracting and 
$\AAd(t)$ is minimal with respect to these properties. 
\end{define}

Next, we define the concept of the pullback $\WWd$-limit.

\begin{define}
For each $A\subseteq X$ and $t\in\R$, we define the pullback $\WW$-limit 
($\bullet=$s, w) of $A$ as 
\begin{equation*}
\WWd(A,t):=\bigcap_{s\le t}\overline{\bigcup_{r\le s}P(t,r)A}^\bullet.
\end{equation*}
\end{define}

Equivalently, we have that $x\in\WWd(A,t)$ if there exist sequences $s_n
\rightarrow-\infty$, $s_n\le t$, $x_n\in P(t,s_n)A$ so that $x_n\rightarrow x$ 
in the $\dd$-metric. We now present some basic properties of $\WWd$.

\begin{lem}
\label{wproperties}
Let $A\subseteq X$ and $t\in\R$. Then,
\begin{enumerate}
\item $\WWd(A,t)$ is $\dd$-closed ($\bullet=$s, w).
\item $\WWs(A,t)\subseteq\WWw(A,t)$.
\item If $\WWw(A,t)$ is strongly compact and uniformly, strongly, pullback 
attracts $A$, then $\WWs(A,t)=\WWw(A,t)$.
\end{enumerate}
\end{lem}

\begin{proof}
Part 1 is obvious from the definition. For part 2, let $x\in\WWs(A,t)$. Then, 
there exists sequences $s_n\le t$, $s_n\rightarrow-\infty$ and $x_n\in P(t,s_n)
A$ with $x_n\xrightarrow{\ds}x$. But then, $x_n\xrightarrow{\dw}x$ and $x\in
\WWw(A,t)$.

Now, suppose that $\WWw(A,t)$ is strongly compact and $\ds$-pullback attracts 
$A$. Let $x\in\WWw(A,t)$. Then, by definition, there are sequences $s_n\le t$, 
$s_n\rightarrow-\infty$ and $x_n\in P(t,s_n)A$ with $x_n\xrightarrow{\dw}x$. 
Since $\WWw(A,t)$ $\ds$-pullback attracts $A$, there exists a sequence $y_n\in 
\WWw(A,t)$ with $\ds(x_n,y_n)\rightarrow0$ as $n\rightarrow\infty$. Because 
$\ds(x_n,y_n)\rightarrow0$, $\dw(x_n,y_n)\rightarrow0$. Since $\WWw(A,t)$ is 
compact, there is some subsequence $y_{n_k}\xrightarrow{\ds}y$ for some $y\in
\WWw(A,t)$. But then, $x_{n_k}\xrightarrow{\ds}y$ which means that 
$x_{n_k}\xrightarrow{\dw}y$. Thus, $y=x$ which means that $x_{n_k}\xrightarrow
{\ds}x$. That is, $x\in\WWs(A,t)$.
\end{proof}

\begin{lem}
\label{minimal}
Let $A(t)$ be a family of $\dd$-closed, $\dd$-pullback attracting sets 
($\bullet=$s, w). Then $\WWd(X,t)\subseteq A(t)$.
\end{lem}

\begin{proof}
Let $x\in\WWd(X,t)$. Then, there exist sequences $s_n\le t$, $s_n\rightarrow
-\infty$, and $x_n\in P(t,s_n)X$ with $x_n\xrightarrow{\dd}x$. Since $A(t)$ is 
$\dd$-pullback attracting, there exists $a_n\in A(t)$ with $\dd(x_n,a_n)
\rightarrow0$ as $n\rightarrow\infty$. But, $x_n\xrightarrow{\dd}x$ which gives 
us that $a_n\xrightarrow{\dd}x$. Since $A(t)$ is $\dd$-closed, $x\in A(t)$.
\end{proof}

Now, we are ready to show that if the $\dd$-pullback attractor exists, then it 
is unique.

\begin{thm}
\label{exists}
If the pullback attractor $\AAd(t)$ exists ($\bullet=$s, w), then 
\begin{equation*}
\AAd(t)=\WWd(X,t).
\end{equation*}
\end{thm}

\begin{proof}
By the above lemmas, $\WWd(X,t)\subseteq\AAd(t)$. Now, let $x\in\AAd(t)
\backslash\WWd(X,t)$. Then, there exists $\ee>0$ and $s_0\le t$ so that 
for $s\le s_0$, 
\begin{equation}
\label{empty}
P(t,s)X\cap B_\bullet(x,\ee)=\emptyset.
\end{equation}
Otherwise, for each $n$ and $t-n\le0$ there exists $s_n\le t-n$ with 
\begin{equation*}
x_n\in P(t,s_n)X\cap B_\bullet(x,1/n)\ne\emptyset.
\end{equation*}
But, then $x_n\xrightarrow{\dd}x$, and $x\in\WWd(X,t)$. This is a 
contradiction. Thus, (\ref{empty}) holds. In this case, $\AAd(t)\backslash 
B_\bullet(x,\ee)$ is a strict subset of $\AAd(t)$ which is $\dd$-closed 
$\dd$-pullback attracting. This contradicts the definition of $\AAd(t)$.
\end{proof}

An immediate consequence of Theorem~\ref{exists} and Lemma~\ref{minimal} 
is the following:

\begin{cor}
\label{existsiff}
The pullback attractor $\AAd(t)$ exists if and only if $\WWd(X,t)$ is a $\dd$-
pullback attracting set.
\end{cor}

Next, we study the structure of $\WWw(A,t)$ for some $A\subseteq X$ and $t\in
\R$.

\begin{thm}
\label{weakattracts}
Let $A\subseteq X$ be such that for each $t\in\R$ and $r\le t$, there is some 
$u\in\EE([r,\infty))$ with $u(t)\in A$. Then, $\WWw(A,t)$ is a nonempty, 
weakly compact set. Moreover, $\WWw(A,t)$ weakly pullback attracts $A$.
\end{thm}

\begin{proof}
Due to the assumptions on $A$, we have that $P(t,r)A\ne\emptyset$. Also, due 
to the fact that $X$ is weakly compact, we have that 
\begin{equation*}
W(s):=\overline{\bigcup_{r\le s}P(t,r)A}^\mathrm{w}
\end{equation*}
is nonempty and weakly compact for each $s\le t$. Moreover, for $s_0\le s_1\le 
t$, $W(s_0)\subset W(s_1)$. Thus, by Cantor's intersection theorem,
\begin{equation*}
\WWw(A,t)=\bigcap_{s\le t}W(s)
\end{equation*}
is a nonempty weakly compact set.

To see that $\WWw(A,t)$ weakly pullback attracts $A$, suppose for 
contradiction that it doesn't. Then, there exists some $\ee>0$ and a sequence 
$s_n\rightarrow-\infty$, $s_n\le t$ with 
\begin{equation*}
P(t,s_n)A\cap B_\mathrm{w}(\WWw(A,t),\ee)^c\ne\emptyset.
\end{equation*}
Therefore,
\begin{equation*}
K_n:=\overline{\bigcup_{r\le s_n}P(t,r)A}^\mathrm{w}\cap B_\mathrm{w}(\WWw(A,t),
\ee)^c\ne\emptyset.
\end{equation*}
Passing to a subsequence if necessary and reindexing, we can assume that the 
$s_n$'s are monotonically decreasing. Thus, we get a decreasing sequence of 
nonempty weakly compact sets. Again, by Cantor's intersection theorem, we have 
that $x\in\cap_n K_n\ne\emptyset$. That is, 
\begin{equation*}
x\in\bigcap_{s_n\le t}\overline{\bigcup_{r\le s_n}P(t,r)A}^\mathrm{w}=
\WWw(A,t)
\end{equation*}
This contradicts the definition of the $K_n$'s.
\end{proof}

Using the above results, we have the following:

\begin{thm}
\label{weakattractorexists}
Every generalized evolutionary system possesses a weak pullback attractor 
$\AAw(t)$. Moreover, if the strong pullback attractor $\AAs(t)$ exists, then 
$\overline{\AAs(t)}^\mathrm{w}=\AAw(t)$.
\end{thm}

\begin{proof}
Due to Theorem~\ref{weakattracts}, $\WWw(X,t)$ is a non-empty weakly closed, 
weakly pullback attracting set. Therefore, by Theorems~\ref{existsiff} and 
\ref{exists}, $\AAw(t)=\WWw(X,t)$ is the weak pullback attractor.

Now, suppose the strong pullback attractor $\AAs(t)$ exists. Then, by 
Theorem~\ref{exists}, $\AAs(t)=\WWs(X,t)$. Then, since $\AAs(t)$ 
strongly pullback attracts $X$, we have that $\AAs(t)$ must weakly pullback 
attract $X$. If not, there is some $\ee>0$ and a sequence $s_n\rightarrow
-\infty$ with 
\begin{equation*}
x_n\in P(t,s_n)X\cap B_\mathrm{w}(\AAs(t),\ee)^c\ne\emptyset.
\end{equation*}
Since $\AAs(t)$ is strongly pullback attracting, there is a sequence $a_n\in 
\AAs(t)$ with $\ds(x_n,a_n)\rightarrow0$. But, then $\dw(x_n,a_n)\rightarrow0$. 
Then, for $n$ sufficiently large, $x_n\in B_\mathrm{w}(\AAs(t),\ee)$ which 
contradicts the choice of $x_n$. In addition, $\overline{\AAs(t)}^\mathrm{w}$ 
is weakly closed and weakly pullback attracting. thus, by Lemma~\ref{minimal}, 
$\AAw(t)=\WWw(X,t)\subseteq \overline{\AAs(t)}^\mathrm{w}$. Finally, by 
Lemma~\ref{wproperties}, $\overline{\AAs(t)}^\mathrm{w}=
\overline{\WWs(X,t)}^\mathrm{w}\subseteq\WWw(X,t)=\AAw(t)$.
\end{proof}

\section{Examples}
\label{examples}

\subsection{A Single Trajectory}
For our first example, let our generalized evolutionary system on an arbitrary 
phase space $X$ consist of a single trajectory $u\in\EE((-\infty,\infty))$ and 
all of its restrictions, $\EE([s,\infty))=\{u|_{[s,\infty)}\}$. Then, we have that 
\begin{equation*}
P(t,s)X=\{u(t)\}
\end{equation*}
is a single point. Therefore, we have that the strong and weak pullback 
attractors both exist. Moreover, 
\begin{equation*}
\AAs(t)=\AAw(t)=\{u(t)\}.
\end{equation*}

\subsection{An Abstract Example on $\ell^2(\Z)$}
For our second example, let $X$ be the unit ball in $\ell^2(\Z)$. Let the 
strong metric on $\ell^2(\Z)$ be the metric induced by the norm on 
$\ell^2(\Z)$. That is, given by 
\begin{equation*}
\ds(x,y)=\norm{x-y}_{\ell^2(\Z)}=\sqrt{\sum_{n\in\Z}(x_n-y_n)^2}.
\end{equation*}
In a similar fashion, $X$ with the weak topology is metrizable using the weak 
metric 
\begin{equation*}
\dw(x,y)=\sum_{k\in\Z}\frac{1}{2^{\abs{k}}}
                  \frac{\abs{x_k-y_k}}{1+\abs{x_k-y_k}}.
\end{equation*}
Now, consider the following complete trajectory on $X$: Let $n\in\Z$. Then, 
let 
\begin{equation*}
u(t):=\frac{(1+n-t)e_n+(t-n)e_{n+1}}
        {\norm{(1+n-t)e_n+(t-n)e_{n+1}}_{\ell^2(\Z)}}
\end{equation*} 
for $t\in[n,n+1]$ and $\{e_n\}$ the standard basis vectors in 
$\ell^2(\infty)$. That is, we interpolate between the basis vectors and 
normalize onto the boundary of the unit ball. Next, let $\EE((-\infty,\infty))
:=\{u(\cdot+r):r\in\R\}$. That is, the above complete trajectory and all of 
its shifts. As in the previous example, we complete our definition of a 
generalized evolutionary system by letting $\EE([s,\infty)):=\{u|_{[s,\infty)}
:u\in\EE((-\infty,\infty))\}$. As we will see in Section~\ref{relate}, this 
turns $\EE$ into an autonomous evolutionary system. Therefore, we find that 
\begin{equation*}
P(t,s)X=\{u(r):r\in\R\}.
\end{equation*}
This is strongly closed but not weakly closed. Thus, we find that 
\begin{equation*}
\AAs(t)=\{u(r):r\in\R\} \quad\mathrm{and}\quad \AAw(t)=
\overline{\AAs(t)}^w=\{u(r):r\in\R\}\cup\{0\}
\end{equation*}
for each $t\in\R$. In particular, the weak and strong pullback attractors are 
not equal.

\subsection{The Heat Equation}
For our next example, let $X$ be the unit ball in $L^2(\R)$. The strong metric 
on $X$ is the metric induced by the norm on $L^2(\R)$. That is, for any $f,g\in
L^2(\R)$,
\begin{equation*}
\ds(f,g):=\norm{f-g}_{L^2(\R)}=\left(\int_\R\abs{f(x)-g(x)}^2\ddx\right)^{1/2}.
\end{equation*}
For the weak metric, we first choose any countable dense subset $\phi_n$ for 
$L^2(\R)$ for $n\in\N$. For example, one could use wavelets as an orthonormal 
basis, as is explained in \cite{L02}. Then, the weak metric on $X$ is given by 
\begin{equation*}
\dw(f,g):=\sum_{k\in\N}\frac{1}{2^k}
  \frac{\abs{\ip{f}{\phi_k}-\ip{g}{\phi_k}}}
        {1+\abs{\ip{f}{\phi_k}-\ip{g}{\phi_k}}}.
\end{equation*}
Now, consider the heat equation on $X$. That is, for some starting time $s\in
\R$, 
\begin{equation}
\label{heateqn}
\left\{
\begin{array}{l}
u_t=u_{xx} \\
u(s)=f(x)
\end{array}
\right.
\end{equation}
for some $f\in X$. Then, using the Fourier transform, 
\begin{equation*}
\hat{f}(\xi):=\int_\R f(x)e^{ix\cdot\xi}\ddx,
\end{equation*}
we find that a solution to (\ref{heateqn}) is given by 
\begin{equation}
\label{heatformula}
\hat{u}(\xi,t)=e^{\xi^2(s-t)}\hat{f}(\xi).
\end{equation}
Note that by Plancherel's theorem, we may work exclusively in Fourier space. 
Define a generalized evolutionary system on $X$ via 
\begin{equation*}
\EE([s,\infty)):=\{u:u\mathrm{~is~a~solution~to~}(\ref{heateqn})\}.
\end{equation*}
We will see that the weak pullback attractor $\AAw(t)$ is given by the single 
point $\{0\}$ for each $t$. On the other hand, the strong pullback attractor 
$\AAs(t)$ does not exist. 

To see this, we first note that for fixed $t\in\R$, $\norm{u(t)}_{L^2(\R)}
\rightarrow0$ as $s\rightarrow-\infty$. This gives us the candidate weak and 
strong pullback attractor $\{0\}$. In the weak metric, this is the pullback 
attractor. On the other hand, by Theorem~\ref{weakattractorexists}, if the 
strong pullback attractor exists, it must be the case that $\AAw(t)=
\overline{\AAs(t)}^w=\{0\}$. So, the only possibility for the strong pullback 
attractor is $\AAs(t)=\{0\}$. However, we fail to have uniform convergence in 
the strong metric. By definition, if $\{0\}$ was the strong pullback attractor, 
then, for any $\ee>0$, there is an $s_0\le t$ with 
\begin{equation*}
P(t,s)X\subseteq B_s(\{0\},\ee)
\end{equation*}
for each $s\le s_0$. So, let $\ee:=1/2$ and let $s_0\le t$ be given. Then, 
consider $f\in L^2(\R)$ with $\norm{f}_{L^2(\R)}=1$ and $\supp(\hat{f})
\subseteq\{\xi:2^{j-1}\le\abs{\xi}\le2^{j+1}\}$ for some $j\in\Z$ to be 
determined later. Then, we see that 
\begin{align*}
\norm{u(\xi,t)}_{L^2(\R)}^2
 &=   \int_\R e^{2\xi^2(s_0-t)}\hat{f}(\xi)^2\dxi \\
 &=   \int_{2^{j-1}\le\abs{\xi}\le2^{j+1}}
         \exp(2\xi^2(s_0-t))\hat{f}(\xi)^2\dxi \\
 &\ge \exp(2\cdot2^{2j-2}(s_0-t))
         \int_{2^{j-1}\le\abs{\xi}\le2^{j+1}}\hat{f}(\xi)^2\dxi \\
 &=   \exp(2^{2j-1}(s_0-t))\norm{f}_{L^2(\R)}^2 \\
 &=   \exp(2^{2j-1}(s_0-t)).
\end{align*}
Therefore, we find that 
\begin{equation*}
\norm{u(\xi,t)}_{L^2(\R)}\ge\exp(2^{2j-2}(s_0-t)).
\end{equation*}
This is greater than or equal to $\ee:=1/2$ provided that 
\begin{equation*}
j\le \frac{1}{2}\left(\log_2\left(\frac{\ln(2)}{t-s_0}\right)+2\right).
\end{equation*}
Therefore, the convergence to $0$ is not uniform, and no strong pullback 
attractor exists.

\subsection{A Phase Space that is not Weakly Compact}
Finally, a simple example showing the importance of the compactness of $X$ in 
the weak topology. Let $X:=\R$ with the weak and strong metrics both given by 
$\ds(x,y):=\dw(x,y):=\abs{x-y}$ for any $x,y\in\R$. For each $s\in\R$, define
\begin{equation*}
\EE([s,\infty)):=\{u(t):=t-s\}.
\end{equation*}
Then, we have that for some $t\in\R$ and some $s\le t$
\begin{equation*}
P(t,s)X=\{u(t):u(s)\in X,u\in\EE([s,\infty))\}=\{t-s\}.
\end{equation*}
But, as $s\rightarrow-\infty$, the limit does not exist. Thus, the weak and 
strong pullback attractors do not exist.

\section{Existence of a Strong Pullback Attractor}
\label{strongPAC}

\begin{define}
A generalized evolutionary system is pullback asymptotically compact if for 
any $t\in\R$, $s_n\rightarrow-\infty$ with $s_n\le t$, and any $x_n\in P(t,s_n)
X$, the sequence $\{x_n\}$ is relatively strongly compact.
\end{define}

\begin{thm}
\label{strong=weak}
Let $\EE$ be pullback asymptotically compact. Let $A\subseteq X$ be so that 
for each $t\in\R$ and $r\le t$, there is some $u\in\EE([r,\infty))$ with $u(t)
\in A$. Then, $\WWs(A,t)$ is a nonempty strongly compact set which strongly 
pullback attracts $A$. Moreover, $\WWs(A,t)=\WWw(A,t)$.
\end{thm}

\begin{proof}
By Theorem~\ref{exists}, our assumptions on $A$ imply that $\WWw(A,t)\ne
\emptyset$. We will see that $\WWw(A,t)$ strongly pullback attracts $A$. 
Suppose it does not. Then, there is some $\ee>0$ and a sequence $s_n\rightarrow
-\infty$ with 
\begin{equation*}
x_n\in P(t,s_n)A\cap B_\mathrm{s}(\WWw(A,t),\ee)^c\ne\emptyset.
\end{equation*}
Since $X$ is pullback asymptotically compact, this sequence has a convergent 
subsequence. After passing to a subsequence and dropping a subindex, we have 
that $x_n\xrightarrow{\ds}x$. But then, $x_n\xrightarrow{\dw}x$. Therefore, by 
the equivalent definition of $\WWw(A,t)$, $x\in \WWw(A,t)$. However, for large 
enough $n$, we must then have $x_n\in B_\mathrm{s}(\WWw(A,t),\ee)^c$ which 
contradicts our choice of $x_n$.

By Lemma~\ref{wproperties}, $\WWs(A,t)\subseteq\WWw(A,t)$. For the other 
inclusion, let $x\in\WWw(A,t)$. By the equivalent definition of $\WWw(A,t)$, 
there are sequences $s_n\rightarrow-\infty$ with $s_n\le t$ and $x_n\in P(t,
s_n)A$ so that $x_n\xrightarrow{\dw}x$. By pullback asymptotic compactness, 
there is a subsequence $\{x_{n_k}\}$ with $x_{n_k}\xrightarrow{\ds}y$ for 
some $y\in X$. But then, $x_{n_k}\xrightarrow{\dw}y$ which gives us that 
$x=y$ and thus, $x_n\xrightarrow{\ds}x$. That is, $x\in\WWs(A,t)$ and 
$\WWs(A,t)=\WWw(A,t)$.

Finally, we establish the strong compactness of $\WWs(A,t)$. So, let $\{x_n
\}$ be any sequence in $\WWs(A,t)$. By the equivalent definition of $\WWs(A,
t)$, there is a corresponding sequence $\{s_k^n\}$ for each $x_n$ with 
$s_k^n\rightarrow-\infty$, $s_k^n\le t$, and $x_k^n\in P(t,s_k^n)A$ so that 
$x_k^n\xrightarrow{\ds}x_n$. Letting $y_n$, $x_n$ be the diagonals of these 
families, we have that 
\begin{equation*}
\ds(y_n,x_n)\rightarrow0 \mathrm{~as~}n\rightarrow\infty.
\end{equation*}
By pullback asymptotic compactness, $\{y_n\}$ is relatively strongly compact. 
Hence, $\{x_n\}$ is also relatively strongly compact. Since $\WWs(A,t)$ is 
closed, the limit of this subsequence lies in $\WWs(A,t)$ giving us that 
$\WWs(A,t)$ is compact.
\end{proof}

Using this result, we have the following existence result for strong pullback 
attractors.

\begin{thm}
\label{strongexists}
If a generalized evolutionary system $\EE$ is pullback asymptotically compact, 
then $\AAw(t)$ is a strongly compact strong pullback attractor.
\end{thm}

\begin{proof}
By Theorem~\ref{strong=weak}, $\WWs(X,t)$ is strongly compact strong pullback 
attracting set with $\WWs(X,t)=\WWw(X,t)=\AAw(t)$, the weak pullback 
attractor. Therefore, by Theorem~\ref{exists} and Corollary~\ref{existsiff}, 
the strong pullback attractor $\AAs(t)$ exists and $\AAs(t)=\WWs(X,t)=\AAw(t)$.
\end{proof}

\section{Invariance and Tracking Properties}
\label{tracking}
Now, we assume that $\EE$ satisfies \iref{A1}. That is,
\begin{center}
\iref{A1}: $\EE([s,\infty))$ is compact in $C([s,\infty);\Xw)$ for each 
$s\in\R$.
\end{center}
Moreover, we introduce the following variation of the mapping $P$: for $A
\subseteq X$ and $s\le t\in\R$
\begin{equation*}
\tilde{P}(t,s)A:=\{u(t):u(s)\in A, u\in\EE((-\infty,\infty))\}.
\end{equation*}

\begin{define}
We say that a family of sets $\BB(t)\subseteq X$ is pullback semi-invariant 
if for each $s\le t\in\R$,
\begin{equation*}
\tilde{P}(t,s)\BB(s)\subseteq\BB(t).
\end{equation*}
We say that $\BB(t)$ is pullback invariant if for $s\le t\in\R$,
\begin{equation*}
\tilde{P}(t,s)\BB(s)=\BB(t).
\end{equation*}
We say that $\BB(t)$ is pullback quasi-invariant if for each $b\in\BB(t)$, 
there is some complete trajectory $u\in\EE((-\infty,\infty))$ with $u(t)=b$ 
and $u(s)\in\BB(s)$ for each $s\le t$.
\end{define}

Note that if $\BB(t)$ is pullback quasi-invariant, then for each $s\le t$, 
\begin{equation}
\label{qinvariance}
\BB(t)\subseteq\tilde{P}(t,s)\BB(s)\subseteq P(t,s)\BB(s).
\end{equation}
Therefore, if $\BB(t)$ is pullback quasi-invariant and pullback 
semi-invariant, then $\BB(t)$ is pullback invariant.

\begin{thm}
\label{a1weakqinvariant}
Let $\EE$ be a generalized evolutionary system satisfying $\iref{A1}$. Then, 
$\WWw(A,t)$ is pullback quasi-invariant for each $A\subseteq X$.
\end{thm}

\begin{proof}
Let $x\in\WWw(A,t)$. Then, there are sequences $s_n\rightarrow-\infty$ with 
$s_n\le t$ and $x_n\in P(t,s_n)A$ so that $x_n\xrightarrow{\dw}x$. Note that 
by passing to a subsequence, we can assume without loss of generality that 
$s_n$ is a monotonically decreasing sequence. Since $x_n\in P(t,s_n)A$, there 
is some $u_n\in\EE([s_n,\infty))$ with $x_n=u_n(t)$, $u_n(s_n)\in A$. Using 
$\iref{A1}$, $\EE([s_n,\infty))$ is compact in $C([s_n,\infty);\Xw)$. 
Moreover, by the definition of $\EE$, 
\begin{equation*}
\{u|_{[s_1,\infty)}:u\in\EE([s_n,\infty))\}\subseteq\EE([s_1,\infty)).
\end{equation*}
Thus, using compactness on $\{u_n|_{[s_1,\infty)}\}$, we can pass to a 
subsequence and drop a subindex obtaining $u^1\in\EE([s_n,\infty))$ so that 
\begin{equation*}
u_n|_{[s_1,\infty)}\rightarrow u^1 \mathrm{~in~}C([s_1,\infty);\Xw).
\end{equation*}
Repeating the argument above with our subsequence, we can find another 
subsequence which, after dropping another subindex, gives us some $u^2\in\EE
([s_2,\infty))$ with 
\begin{equation*}
u_n|_{[s_2,\infty)}\rightarrow u^2 \mathrm{~in~}C([s_2,\infty);\Xw).
\end{equation*}
Note that, by construction, $u^2|_{[s_1,\infty)}=u^1$. Continuing, 
inductively, we get $u^k\in\EE([s_k,\infty))$ with 
\begin{equation*}
u_n|_{[s_k,\infty)}\rightarrow u^k \mathrm{~in~}C([s_k,\infty);\Xw).
\end{equation*}
and $u^k|_{[s_{k-1},\infty)}=u^{k-1}$. A standard diagonalization process 
gives us some subsequence of $u_n$ and $u\in\EE((-\infty,\infty))$ so that 
$u|_{[-T,\infty)}\in\EE([-T,\infty))$ and $u_n\rightarrow u$ in $C([-T,\infty);
\Xw)$ for any $T>0$. That is, $u\in\EE((-\infty,\infty))$, by definition. 

Note that $u(t)=x$, by construction. Now, let $s\le t$. Then, $u_n(s)
\xrightarrow{\dw}u(s)$. By definition, since $u_n(s_n)\in A$, we then get that 
$u_n(s)\in P(s,s_n)A$ for $s_n\le s$ ($n$ sufficiently large). Hence, $u(s)\in 
\WWw(A,s)$ and $\WWw(A,t)$ is pullback quasi-invariant.
\end{proof}

This characterization of $\WWw(A,t)$ gives us the following important 
consequences.

\begin{cor}
\label{a1weak=strong}
Let $\EE$ be a generalized evolutionary system satisfying $\iref{A1}$. Let $A
\subseteq X$ be such that $\WWw(A,t)\subseteq A$ for each $t\in\R$. Then, 
$\WWw(A,t)=\WWs(A,t)$.
\end{cor}

\begin{proof}
By Theorem~\ref{a1weakqinvariant}, we have that $\WWw(A,t)$ is pullback 
quasi-invariant. Thus, by (\ref{qinvariance}), we have that $\WWw(A,t)
\subseteq P(t,s)\WWw(A,s)$ for each $s\le t$. By assumption, $\WWw(A,s)
\subseteq A$, thus $\WWw(A,t)\subseteq P(t,s)A$ for each $s\le t$. Therefore, 
$\WWw(A,t)\subseteq\WWs(A,t)$. On the other hand, by Lemma~\ref{wproperties}, 
$\WWs(A,t)\subseteq\WWw(A,t)$. That is, $\WWw(A,t)=\WWs(A,t)$.
\end{proof}

The following result is a direct result of Corollary~\ref{a1weak=strong} and 
Theorem~\ref{exists}.

\begin{cor}
Let $\EE$ be a generalized evolutionary system satisfying $\iref{A1}$. Then, 
if the strong pullback attractor $\AAs(t)$ exists, $\AAs(t)=\AAw(t)$, the weak 
pullback attractor.
\end{cor}

In fact, we get a new characterization of pullback invariance for a weakly 
closed set $A$.

\begin{thm}
\label{inviffqinvandsinv}
Let $\EE$ be a generalized evolutionary system satisfying $\iref{A1}$. Then, 
for a family of weakly closed subsets $\BB(t)\subseteq X$, $\BB(t)$ is 
pullback invariant if and only if $\BB(t)$ is pullback semi-invariant and 
pullback quasi-invariant.
\end{thm}

\begin{proof}
If $\BB(t)$ is pullback semi-invariant and pullback quasi-invariant, then by 
the definition of pullback semi-invariant and (\ref{qinvariance}), $\BB(t)$ is 
pullback invariant.

For the other direction, assume $\BB(t)$ is pullback invariant. Then, we have 
that $\BB(t)$ is clearly pullback semi-invariant. To see pullback 
quasi-invariance, let $b\in\BB(t)$. Then, by pullback invariance, we can 
construct a monotonically decreasing sequence $s_n\rightarrow-\infty$ and find 
$u_n\in\EE([s_n,\infty))$ with $u_n(s_n)\in\BB(s_n)$ and $u_n(t)=b$. As in 
Theorem~\ref{a1weakqinvariant}, we can find a subsequence which, after dropping 
a subindex, is so that $u_n\rightarrow u$ for some $u\in\EE((-\infty,\infty))$ 
in the sense of $C([-T,\infty);\Xw)$ for each $T>0$. Moreover, $u(t)=b$ and 
$u(s)\in\BB(s)$ for each $s\le t$ since each $\BB(s)$ is weakly closed. 
Therefore, $\BB(t)$ is pullback quasi-invariant.
\end{proof}

Let $\II(t)$ be a family of subsets of $X$ given by 
\begin{equation*}
\II(t):=\{u(t):u\in\EE((-\infty,\infty))\}.
\end{equation*}
Then, $\II(t)$ is both pullback semi-invariant and pullback quasi-invariant. 
Moreover, $\II(t)$ contains every pullback quasi-invariant and every pullback 
invariant set. Thus, by Theorem~\ref{a1weakqinvariant}, 
\begin{equation*}
\WWw(A,t)\subseteq\II(t)
\end{equation*}
for each $t\in\R$ and each $A\subseteq X$.

Now, we will show that $\WWw(A,t)$ contains all the asymptotic behavior (as 
the initial time goes to $-\infty$) of every trajectory starting in $A$, 
provided $\iref{A1}$ holds.

\begin{thm}[Weak pullback tracking property]
\label{weaktracking}
Let $\EE$ be a generalized evolutionary system satisfying $\iref{A1}$, and let 
$A\subseteq X$. Then, for each $\ee>0$ and each $t\in\R$, there is some $s_0:=
s_0(\ee,t)\le t$ so that for $s'<s_0$ and $u\in\EE([s',\infty))$ with $u(s')\in 
A$ satisfies
\begin{equation*}
\mathrm{d}_{C([s',\infty);\Xw)}(u,v)<\ee
\end{equation*}
for some $v\in\EE((-\infty,\infty))$ with $v(s)\in\WWw(A,s)$ for each $s\le t$.
\end{thm}

\begin{proof}
For contradiction, suppose not. Then, there exists $\ee>0$ and sequences $s_n
\le t$ with $s_n\rightarrow-\infty$, $u_n\in\EE([s_n,\infty))$ with the 
property that $u_n(s_n)\in A$ and 
\begin{equation}
\label{weakcontradiction}
\mathrm{d}_{C([s_n,\infty);\Xw)}(u_n,v)\ge\ee
\end{equation}
for each $n$ and each $v\in\EE((-\infty,\infty))$ with $v(s)\in\WWw(A,s)$ for 
$s\le t$. As in the proof of Theorem~\ref{a1weakqinvariant}, we find a $u\in
\EE((-\infty,\infty))$ and a subsequence which after dropping the subindex can 
be written as $u_n$ with $u_n\rightarrow u$ in $C([-T,\infty);\Xw)$ for each 
$T>0$. In particular, $u(s)\in\WWw(A,s)$ for each $s\le t$. In particular, for 
large enough $n$, $\mathrm{d}_{C([s_n,\infty);\Xw)}(u_n,u)<\ee$ which 
contradicts (\ref{weakcontradiction}).
\end{proof}

\begin{thm}[Strong pullback tracking property]
\label{strongtracking}
Let $\EE$ be a pullback asymptotically compact generalized evolutionary system 
satisfying $\iref{A1}$ and let $A\subseteq X$. Then, for each $\ee>0$, $t\in
\R$, and $T>0$, there is some $s_0:=s_0(\ee,t,T)\le t$ so that for $s'<s_0$ 
and each $u\in\EE([s',\infty))$ with $u(s')\in A$, we have 
\begin{equation*}
\mathrm{d}_\mathrm{s}(u(\hat{s}),v(\hat{s}))<\ee
\end{equation*}
for each $\hat{s}\in[s',s'+T]$ and some $v\in\EE((-\infty,\infty))$ so that 
$v(s)\in\WWs(A,s)$ for each $s\le t$.
\end{thm}

\begin{proof}
Again, suppose not. Then, there is some $\ee>0$, $T>0$, and sequences $s_n\le 
t$ with $s_n\rightarrow-\infty$, $u_n\in\EE([s_n,\infty))$ so that $u_n(s_n)
\in A$ and 
\begin{equation}
\label{strongcontradiction}
\sup_{\hat{s}\in[s_n,s_n+T]}\mathrm{d}_{\mathrm{s}}(u_n(\hat{s}),v(\hat{s}))\ge
\ee
\end{equation}
for each $n$ and each $v\in\EE((-\infty,\infty))$ with $v(s)\in\WWs(A,s)$ for 
each $s\le t$.

By Theorem~\ref{weaktracking}, there exists a sequence $v_n\in\EE((-\infty,
\infty))$ with $v_n(s)\in\WWw(A,s)$ for $s\le t$ so that 
\begin{equation}
\label{strongcontra2}
\lim_{n\rightarrow\infty}\sup_{\hat{s}\in[s_n,s_n+T]}\mathrm{d}_\mathrm{w}
(u_n(\hat{s}),v_n(\hat{s}))=0.
\end{equation}
Using the pullback asymptotic compactness of $\EE$ we get that $\WWs(A,t')=
\WWw(A,t')$ for all $t'\in\R$. Then, by (\ref{strongcontradiction}), there is 
a sequence $\hat{s_n}\in[s_n,s_n+T]$ so that 
\begin{equation}
\label{strongcontra3}
\mathrm{d}_\mathrm{s}(u_n(\hat{s_n}),v_n(\hat{s_n}))\ge\ee/2.
\end{equation}
Again, using the pullback asymptotic compactness of $\EE$, the sequences 
$\{u_n(\hat{s_n})\}$ and $\{v_n(\hat{s_n})\}$ have convergent subsequences. 
So, passing to a subsequence and dropping a subindex, we have that 
$u_n(\hat{s_n})\xrightarrow{\ds}x$, $v_n(\hat{s_n})\xrightarrow{\ds}y$ for 
some $x,y\in X$. By (\ref{strongcontra2}), $x=y$ contradicting 
(\ref{strongcontra3}).
\end{proof}

Next, we use the above tracking properties to show that the weak pullback 
attractor $\AAw(t)=\II(t)$ if the generalized evolutionary system $\EE$ 
satisfies $\iref{A1}$.

\begin{thm}
\label{weak=i}
Let $\EE$ be a generalized evolutionary system satisfying $\iref{A1}$. Then, 
the weak pullback attractor $\AAw(t)=\II(t)$, and $\AAw(t)$ is the maximal 
pullback quasi-invariant and maximal pullback invariant subset of $X$. 
Moreover, for each $\ee>0$ and $t\in\R$ there is some $s_0:=s_0(\ee,t)\le t$ so 
that for $s'<s_0$ and every trajectory $u\in\EE([s',\infty))$ has 
\begin{equation*}
\mathrm{d}_{C([s',\infty);\Xw)}(u,v)<\ee
\end{equation*}
for some complete trajectory $v\in\EE((-\infty,\infty))$.
\end{thm}

\begin{proof}
Since $\AAw(t)=\WWw(X,t)$ and $\WWw(X,t)$ is pullback quasi-invariant by 
Theorem~\ref{a1weakqinvariant}, we have by (\ref{qinvariance}) that $\AAw(t)
\subseteq\II(t)$. For the other inclusion, let $u(t)\in\II(t)$. Suppose $u(t)
\not\in\AAw(t)$. Then, since $\AAw(t)$ is weakly closed, there is some $\ee>0$ 
and $B_\mathrm{w}(u(t),\ee)$ so that $\AAw(t)\cap B_\mathrm{w}(u(t),\ee)=
\emptyset$. By the weak pullback tracking property on $\AAw(t)=\WWw(X,t)$, 
there is some $s'$ so that for any $\hat{u}\in\EE([s',\infty))$, 
\begin{equation*}
\mathrm{d}_{C([s',\infty);\Xw)}(\hat{u},v)<\ee
\end{equation*}
for some $v\in\EE((-\infty,\infty))$ with $v(s)\in\WWw(X,s)$ for each $s\le t$. 
In particular, $u\in\EE([s',\infty))$ for each $s'<t$. Thus, there is some 
$v\in\EE((-\infty,\infty))$ so that $v(t)\in\WWw(X,t)$ and 
\begin{equation*}
\dw(u(t),v(t))\le\mathrm{d}_{C([s',\infty);\Xw)}(u,v)<\ee.
\end{equation*}
This contradicts that $\AAw(t)\cap B_\mathrm{w}(u(t),\ee)=\emptyset$. The rest 
of the theorem follows from Theorem~\ref{weaktracking}.
\end{proof}

Putting together this result as well as Theorem~\ref{strongexists} and 
Theorem~\ref{strongtracking}, we have the following corollary.

\begin{cor}
\label{strong=i}
Let $\EE$ be a pullback asymptotically compact generalized evolutionary system 
satisfying $\iref{A1}$. Then, the strong pullback attractor $\AAs(t)=\II(t)$ 
and $\AAs(t)$ is the maximal pullback invariant and maximal pullback 
quasi-invariant set. Moreover, for each $\ee>0$, $t\in\R$, and $T>0$, there 
is some $s_0:=s_0(\ee,t,T)\le t$ so that for $s'<s_0$, every trajectory $u\in
\EE([s',\infty))$ satisfies
\begin{equation*}
\ds(u(s),v(s))<\ee
\end{equation*}
for each $s\in[s',s'+T]$ and some complete trajectory $v\in\EE((-\infty,
\infty))$.
\end{cor}

\section{Energy Inequality}
\label{energy}
In this section, we assume that our generalized evolutionary system satisfies 
properties $\iref{A2}$ and $\iref{A3}$. That is, for $\iref{A2}$, we let $X$ 
be a set in some Banach space $H$ satisfying the Radon-Riesz Property with norm 
$\abs{\cdot}$ so that $\ds(x,y)=\abs{x-y}$ for each $x,y\in X$, and assume that 
$\dw$ induces the weak topology on $X$. Assume that for each $\ee>0$ and each 
$s\in\R$ there is a $\delta:=\delta(\ee,s)$ so that for every $u\in
\EE([s,\infty))$ and $t>s\in\R$ that 
\begin{equation*}
\abs{u(t)}\le\abs{u(t_0)}+\ee
\end{equation*}
for $t_0$ a.e. in $(t-\delta,t)$. For $\iref{A3}$ we assume that if $u,u_n\in
\EE([s,\infty))$ with $u_n\rightarrow u$ in $C([s,t];\Xw)$ for some $s\le t\in
\R$, then, $u_n(t_0)\xrightarrow{\ds}u_n(t_0)$ for a.e. $t_0\in[s,t]$.

\begin{thm}
\label{strongconv}
Let $\EE$ be a generalized evolutionary system satisfying $\iref{A2}$ and 
$\iref{A3}$. Let $u_n\in\EE([s,\infty))$ be so that $u_n\rightarrow u$ in 
$C([s,t];\Xw)$ for some $u\in\EE([s,\infty))$. If $u(t)$ is strongly 
continuous at some $t^*\in(s,t)$, then $u_n(t^*)\xrightarrow{\ds}u(t^*)$.
\end{thm}

\begin{proof}
By $\iref{A3}$, there is a set $E\subset[s,t]$ of measure zero so that 
$u_n(t_0)\xrightarrow{\ds}u(t_0)$ on $[s,t]\backslash E$. Let $\ee>0$. By the 
energy inequality $\iref{A2}$ and the strong continuity of $u(t)$, there is 
some $t_0\in[s,t^*)\backslash E$ so that 
\begin{equation*}
\abs{u_n(t^*)}\le\abs{u_n(t_0)}+\ee/2,\quad\abs{u(t_0)}\le\abs{u(t^*)}+\ee/2,
\end{equation*}
for each $n$. Taking the upper limit, we then have that 
\begin{equation*}
\limsup_{n\rightarrow\infty}\abs{u_n(t^*)}
  \le\limsup_{n\rightarrow\infty}\abs{u_n(t_0)}+\ee/2
  =\abs{u(t_0)}+\ee/2
  \le\abs{u(t^*)}+\ee.
\end{equation*}
Letting $\ee\rightarrow0$, we have that 
\begin{equation*}
\limsup_{n\rightarrow\infty}\abs{u_n(t^*)}\le\abs{u(t^*)}.
\end{equation*}
Since $u_n(t^*)\xrightarrow{\dw}u(t^*)$, by assumption, we know that 
$\liminf_{n\rightarrow\infty}\abs{u_n(t^*)}\ge\abs{u(t^*)}.$ Thus, $\lim_{n
\rightarrow\infty}\abs{u_n(t^*)}=\abs{u(t^*)}$, and, using the 
Radon-Riesz property, we have that $u_n(t^*)\xrightarrow{\ds}u(t^*)$.
\end{proof}

\begin{thm}
\label{123implyPAC}
Let $\EE$ be a generalized evolutionary system satisfying $\iref{A1}$, 
$\iref{A2}$, and $\iref{A3}$. If $\EE((-\infty,\infty))\subseteq 
C((-\infty,\infty);\Xs)$, then $\EE$ is pullback asymptotically compact.
\end{thm}

\begin{proof}
Let $s_n\rightarrow-\infty$, $s_n\le t$ for some $t\in\R$ and $x_n\in 
P(t,s_n)X$. Since $X$ is weakly compact, we can pass to a subsequence and drop 
a subindex to assume that $x_n\xrightarrow{\dw}x$ for some $x\in X$. 

Next, since $x_n\in P(t,s_n)X$, there is some $u_n\in\EE([s_n,\infty))$ 
with $u_n(t)=x_n$ for each $n$. Using $\iref{A1}$ and the usual 
diagonalization process, we can pass to a subsequence and drop a subindex 
to find that $u_n\rightarrow u$ in $C((-\infty,\infty),\Xw)$ for some $u\in
\EE((-\infty,\infty))\subseteq C((-\infty,\infty);\Xs)$. Since $u$ is strongly 
continuous at $t$, Theorem~\ref{strongconv} implies that $u_n(t)=x_n
\xrightarrow{\ds}x=u(t)$. Therefore, $\EE$ is pullback asymptotically compact.
\end{proof}

Together with Theorem~\ref{strongexists}, we have the following:

\begin{cor}
\label{123implystrong}
let $\EE$ be a generalized evolutionary system satisfying $\iref{A1}$, 
$\iref{A2}$, and $\iref{A3}$. If every complete trajectory is strongly 
continuous, then $\EE$ possesses a strongly compact, strong pullback 
attractor $\AAs(t)$.
\end{cor}

In fact, following the proofs of Theorem~\ref{strong=weak} and 
Theorem~\ref{123implyPAC}, we have the following generalization.

\begin{thm}
\label{123omegalimits}
Let $\EE$ be a generalized evolutionary system satisfying $\iref{A1}$, $\iref
{A2}$, and $\iref{A3}$. Let $A\subseteq X$ be such that for each $s\in\R$ 
there exists some $u\in\EE([s,\infty))$ with $u(t)\in A$. Assume that $u$ is 
strongly continuous at $t$ for each $u\in\EE((-\infty,\infty))$ with $u(t)\in
\WWw(A,t)$. Then, $\WWw(A,t)$ is a nonempty, strongly compact set that strongly 
pullback attracts $A$. Moreover, $\WWs(A,t)=\WWw(A,t)$.
\end{thm}

\section{Pullback Attractors for Evolutionary Systems}
\label{relate}
\subsection{Autonomous Case}

We will begin with the definitions and major results for autonomous 
evolutionary systems as given in \cite{C09}, \cite{CF06}. Note that $X$ has 
the same structure as it had in Section~\ref{setup}. That is, $X$ is endowed 
with two metrics $\ds$ known as the strong metric and $\dw$ known as the 
weak metric so that $X$ is $\dw$-compact and every $\ds$-convergent sequence 
is also $\dw$-convergent.

\begin{define}\cite{C09}
\label{evolutionarysystem}
A map $\EE$ that associates to each $I\in\TT$ a subset $\EE(I)\subseteq\FF(I)$ 
will be called an evolutionary system if the following conditions are satisfied:
\begin{enumerate}
\item $\EE([0,\infty))\ne\emptyset$.
\item $\EE(I+s)=\{u(\cdot):u(\cdot+s)\in\EE(I)\}$ for all $s\in\R$.
\item $\{u(\cdot)|_{I_2}:u(\cdot)\in\EE(I_1)\}\subseteq\EE(I_2)$ for all pairs 
$I_1,I_2\in\TT$, so that $I_2\subseteq I_1$.
\item $\EE((-\infty,\infty))=\{u(\cdot):u(\cdot)|_{[T,\infty)}\in\EE([T,\infty))
\forall T\in\R\}$.
\end{enumerate}
\end{define}

As with a generalized evolutionary system, $\EE(I)$ is referred to as the set of 
all trajectories on the time interval $I$. Trajectories in $\EE((-\infty,\infty))$ are 
known as complete. Let $\PP(X)$ be the set of all subsets of $X$. For each 
$t\ge 0$, the map
\begin{equation*}
R(t):\PP(X)\rightarrow\PP(X)
\end{equation*}
is defined by
\begin{equation*}
R(t)A:=\{u(t):u(0)\in A,u\in\EE([0,\infty))\}\mathrm{~for~}A\subseteq X.
\end{equation*}
By the definitions of $\EE$ and $R(t)$, $R$ has the following property for each 
$A\subseteq X$, $t,s\ge 0$,
\begin{equation*}
R(t+s)A\subseteq R(t)R(s)A.
\end{equation*}

\begin{define}\cite{C09}
A set $\AAd\subseteq X$ is a $\dd$-global attractor ($\bullet=$ s or w) if 
$\AAd$ is a minimal set which is
\begin{enumerate}
\item $\dd$-closed.
\item $\dd$-attracting: for any $B\subseteq X$ and any $\ee>0$, there is a 
$t_0:=t_0(B,\ee)$ so that 
\begin{equation*}
R(t)B\subseteq B_\bullet(\AAd,\ee):=\{u:\inf_{x\in\AAd}\dd(u,x)<\ee\}
      \mathrm{~for~all~}t\ge t_0.
\end{equation*}
\end{enumerate}
\end{define}

\begin{define}\cite{C09}
The $\ww_\bullet$-limit ($\bullet =$ s or w) of a set $A\subseteq X$ is
\begin{equation*}
\ww_\bullet(A):=\bigcap_{T\ge0}\overline{\bigcup_{t\ge T}R(t)A}^\bullet.
\end{equation*}
\end{define}

Equivalently, $x\in \ww_\bullet(A)$ if there exist sequences $t_n\rightarrow\infty$ 
as $n\rightarrow\infty$ and $x_n\in R(t_n)A$, such that $x_n\xrightarrow{\dd}x$ 
as $n\rightarrow\infty$. 

To extend the notion of invariance from a semiflow to an evolutionary 
system, the following mapping is used for $A\subseteq X$ and $t\in\R$:
\begin{equation*}
\tilde{R}(t)A:=\{u(t):u(0)\in A, u\in\EE((-\infty,\infty))\}.
\end{equation*}

\begin{define}\cite{C09}
A set $A\subseteq X$ is positively invariant if for each $t\ge0$,
\begin{equation*}
\tilde{R}(t)A\subseteq A.
\end{equation*}
We say that $A$ is invariant if for each $t\ge0$,
\begin{equation*}
\tilde{R}(t)(A)=A.
\end{equation*}
$A$ is quasi-invariant if for every $a\in A$, there exists a complete 
trajectory $u\in\EE((-\infty,\infty))$ with $u(0)=a$ and $u(t)\in A$ for all 
$t\in\R$.
\end{define}

\begin{define}\cite{C09}
The evolutionary system $\EE$ is asymptotically compact if for any $t_k
\rightarrow\infty$ and any $x_k\in R(t_k)X$, the sequence $\{x_k\}$ is 
relatively strongly compact.
\end{define}

Here are other assumptions that are imposed on $\EE$.
\begin{enumerate}
\refitem{B1} $\EE([0,\infty))$ is a compact set in $C([0,\infty);
\Xw)$.
\refitem{B2} Assume that $X$ is a set in some Banach space $H$ satisfying the 
Radon-Riesz property with the norm denoted by $\abs{\cdot}$, so that $\ds(x,y)=
\abs{x-y}$ for $x,y\in X$ and $\dw$ induces the weak topology on $X$. Assume 
also that for any $\ee>0$, there exists $\delta:=\delta(\ee)$, such that for 
every $u\in\EE([0,\infty))$ and $t>0$,
\begin{equation*}
\abs{u(t)}\le\abs{u(t_0)}+\ee,
\end{equation*}
for $t_0$ a.e. in $(t-\delta,t)$.
\refitem{B3} Let $u,u_n\in\EE(([0,\infty))$, be so that $u_n\rightarrow u$ in 
$C([0,T];\Xw)$ for some $T>0$. Then, $u_n(t)\rightarrow u(t)$ strongly a.e. in 
$[0,T]$.
\end{enumerate}

\begin{thm}\cite{C09}
\label{globalattractorresults}
Let $\EE$ be an evolutionary system. Then,
\begin{enumerate}
\item If the $\dd$-global attractor $\AAd$ exists, then $\AAd=\wwd(X)$.
\item The weak global attractor $\AAw$ exists.
\end{enumerate}
Furthermore, if $\EE$ satisfies $\iref{B1}$, then
\begin{enumerate}
\addtocounter{enumi}{2}
\item $\AAw=\ww_\mathrm{w}(X)=\ww_\mathrm{s}(X)=\{u_0:u_0=u(0)\mathrm{~for~
some~}u\in\EE((-\infty,\infty))\}$.
\item $\AAw$ is the maximal invariant and maximal quasi-invariant set.
\item (Weak uniform tracking property) For any $\ee>0$, there exists a $t_0:=
t_0(\ee)$, so that for any $t>t_0$, every trajectory $u\in\EE([0,\infty))$ 
satisfies $\mathrm{d}_{C([t,\infty);\Xw)}(u,v)<\ee$, for some complete 
trajectory $v\in\EE((-\infty,\infty))$.
\end{enumerate}
\end{thm}

\begin{thm}\cite{C09}
\label{globalattractorAC}
Let $\EE$ be an asymptotically compact evolutionary system. Then,
\begin{enumerate}
\item The strong global attractor $\AAs$ exists, it is strongly compact, 
and $\AAs=\AAw$.
\end{enumerate}
Furthermore, if $\EE$ satisfies $\iref{B1}$, then
\begin{enumerate}
\addtocounter{enumi}{1}
\item (Strong uniform tracking property) for any $\ee>0$ and $T>0$, there 
exists $t_0:=t_0(\EE,T)$, so that for any $t^*>t_0$, every trajectory $u\in
\EE([0,\infty))$ satisfies $\ds(u(t),v(t))<\ee$, for all $t\in[t^*,t^*+T]$, for 
some complete trajectory $v\in\EE((-\infty,\infty))$.
\end{enumerate}
\end{thm}

\begin{thm}\cite{C09}
\label{123giveAC}
Let $\EE$ be an evolutionary system satisfying $\iref{B1}$, 
$\iref{B2}$, and $\iref{B3}$ and so that every complete trajectory is 
strongly continuous. Then, $\EE$ is asymptotically compact.
\end{thm}

Now, we consider the existence of the pullback attractor in the context of 
an evolutionary system. Note that every evolutionary system is also a 
generalized evolutionary system. In this case, we have the following 
relationship for the set functions $P$ and $R$: Let $s\le t\in\R$, and let 
$A\subseteq X$, then 
\begin{equation*}
P(t,s)A=P(t-s,0)A=R(t-s)A.
\end{equation*}
The following results can also be easily verified:
\begin{enumerate}
\item A set $A$ is $\dd$-attracting if and only if the family of sets $A(t):=
A$ for all $t\in\R$ is $\dd$-pullback attracting.
\item $\EE$ is asymptotically compact if and only if $\EE$ is pullback 
asymptotically compact.
\item $\EE$ satisfies $\iref{B1}$ if and only if $\EE$ satisfies $\iref{A1}$.
\item $\EE$ satisfies $\iref{B2}$ if and only if $\EE$ satisfies $\iref{A2}$.
\item $\EE$ satisfies $\iref{B3}$ if and only if $\EE$ satisfies $\iref{A3}$.
\end{enumerate}
Moreover, for $B\subseteq X$ and $B(t):=B$ for all $t\in\R$, we have that 
the following invariance relations:
\begin{enumerate}
\addtocounter{enumi}{5}
\item $B$ is positively invariant if and only if $B(t)$ is pullback 
semi-invariant.
\item $B$ is invariant if and only if $B(t)$ is pullback invariant.
\item If $B$ is quasi-invariant, then $B(t)$ is pullback quasi-invariant.
\end{enumerate}
This gives us the following characterization of the $\wwd$-limit and 
$\WWd$-limit sets.

\begin{thm}
\label{ww=WW}
Let $\EE$ be an evolutionary system. Let $t\in\R$ and $A\subseteq X$ then, 
\begin{equation*}
\WWd(A,t)=\wwd(A) \quad \mathrm{(\bullet=s,w)}.
\end{equation*}
\end{thm}

\begin{proof}
Let $x\in\WWd(A,t)$. Then, there exist sequences $s_n\rightarrow-\infty$, 
$s_n\le t$, and $x_n\in P(t,s_n)A$ so that $x_n\xrightarrow{\dd}x$. Then, 
$(t-s_n)\rightarrow\infty$ for $(t-s_n)\ge0$, and $x_n\in P(t,s_n)A=P(t-s_n,0)A
=R(t-s_n)A$. That is, $x\in\wwd(A)$.

Now, let $x\in\wwd(A)$. Then, there exist sequences $t_n\rightarrow\infty$, 
$t_n\ge 0$, and $x_n\in R(t_n)A$ so that $x_n\xrightarrow{\dd}x$. But, $t_n=
t-(t-t_n)$. Therefore, 
\begin{equation*}
x_n\in R(t_n)A=R(t-(t-t_n))A=P(t-(t-t_n),0)A=P(t,t-t_n)A
\end{equation*}
for $t-t_n\le t$ and $t-t_n\rightarrow\infty$. That is, $x\in\WWd(A,t)$.
\end{proof}

Using this result as well as Theorem~\ref{globalattractorresults}, 
Theorem~\ref{exists}, and Theorem~\ref{weakattractorexists}, we have the 
following corollary.

\begin{thm}
\label{GA=PA}
Let $\EE$ be an evolutionary system. Then, the weak global attractor $\AAw$ and 
the weak pullback attractor $\AAw(t)$ exist, and $\AAw=\AAw(t)$ for each $t\in
\R$. Moreover, the strong global attractor $\AAs$ exists if and only if the 
strong pullback attractor $\AAs(t)$ exists, and $\AAs=\AAs(t)$.
\end{thm}

\begin{proof}
Using Theorem~\ref{globalattractorresults}, we know that the weak global 
attractor $\AAw$ exists, and $\AAw=\www(X)$.  By Theorem~\ref{exists} and 
Theorem~\ref{weakattractorexists}, the weak pullback attractor $\AAw(t)$ 
exists and $\AAw(t)=\WWw(X,t)$. By Theorem~\ref{ww=WW}, we have that 
\begin{equation*}
\AAw=\www(X)=\WWw(X,t)=\AAw(t).
\end{equation*}

Now, suppose the strong global attractor $\AAs$ exists. Then, as in the above 
section, we have that 
\begin{equation*}
\AAs=\wws(X)=\WWs(X,t).
\end{equation*}
But, then $\WWw(X,t)$ is $\ds$-attracting which means that it is $\ds$-pullback 
attracting. Therefore, the strong pullback attractor $\AAs(t)$ exists, and 
$\AAs(t)=\AAs$. An analogous argument shows that if the strong pullback 
attractor $\AAs(t)$ exists, then the strong global attractor $\AAs$ exists and 
$\AAs=\AAs(t)$.
\end{proof}

Furthermore, if $\EE$ is asymptotically compact, then $\EE$ is pullback 
asymptotically compact, and the strong global attractor $\AAs$ and strong 
pullback attractor $\AAs(t)$ both exist, and $\AAs(t)=\AAs$.

\subsection{Nonautonomous Case}

The modern theory of uniform attractors (using a symbol space across which 
attraction is uniform) was first introduced by Chepyzhov and Vishik. They 
applied this framework to the 2D Navier-Stokes equations with an appropriate 
forcing term. For more information on this theory, see \cite{CV02}. Later, a 
weak uniform attractor was proved for the 3D Navier-Stokes equations by 
Kapustyan and Valero \cite{KV07}. Using the framework of evolutionary systems, 
Cheskidov and Lu added a structure theorem and tracking properties of the 
uniform attractor \cite{CL12}. We follow the closely-related framework in 
\cite{CL12} to compare the structure of the weak uniform attractor to the 
weak pullback attractor.

Following the theory of \cite{CV02} the concept of symbols is introduced. So, 
let $\Sig$ be a parameter set and $\{T(s)\}_{s\ge0}$ be a family of operators 
acting on $\Sig$ satisfying $T(s)\Sig=\Sig$, for each $s\ge0$. Any element 
$\sig\in\Sig$ will be called a (time) symbol and $\Sig$ will be called the 
(time) symbol space. For instance, in many applications $\{T(s)\}$ is the 
translation semigroup and $\Sig$ is the translation family of time dependent 
items of the system being considered or its closure in some appropriate 
topological space.

\begin{define}\cite{CL12}
\label{nonautevolutionarysystem}
A family of maps $\EE_\sig$, $\sig\in\Sig$ which for each $\sig\in\Sig$ 
associates to each $I\in\TT$ a subset $\EE_\sig(I)\subseteq\FF(I)$ will be 
called a nonautonomous evolutionary system if the following conditions are 
satisfied:
\begin{enumerate}
\item $\EE_\sig([\tau,\infty))\ne\emptyset$ for each $\tau\in\R$.
\item $\EE_\sig(I+s)=\{u(\cdot):u(\cdot+s)\in\EE_{T(s)\sig}(I)\}$ for each 
$s\ge0$.
\item $\{u(\cdot)|_{I_2}:u(\cdot)\in\EE_\sig(I_1)\}\subseteq\EE_\sig(I_2)$ for 
each $I_1,I_2\in\TT$, $I_2\subseteq I_1$.
\item $\EE_\sig((-\infty,\infty))=\{u(\cdot):u(\cdot)|_{[\tau,\infty)}\in
\EE_\sig([\tau,\infty))\forall\tau\in\R\}$.
\end{enumerate}
\end{define}

Analogous to our previous definitions, $\EE_\sig(I)$ is called the set of all 
trajectories with respect to the symbol $\sig$ on the time interval $I$. 
Trajectories in $\EE_\sig((-\infty,\infty))$ are called complete with respect 
to $\sig$. Note that if we fix any symbol $\sig\in\Sig$ in a nonautonomous 
evolutionary system $\EE_\sig$, then we obtain a generalized evolutionary 
system. On the other hand, if we let $\Sig:=\R$ with $T(s)t:=t+s$, the 
translation semigroup, as well as 
\begin{equation*}
\EE_t([T,\infty)):=\{u(\cdot):u(\cdot-t)\in\EE([T+t,\infty))\},
\end{equation*}
we obtain a nonautonomous evolutionary system from a given generalized 
evolutionary system. 

For every $t\ge\tau$, $\tau\in\R$, $\sig\in\Sig$, the map 
\begin{equation*}
R_\sig(t,\tau):\PP(X)\rightarrow\PP(X)
\end{equation*}
is defined by 
\begin{equation}
\label{shiftRinunion}
R_\sig(t,\tau)A := \{u(t):u(\tau)\in A, u\in\EE_\sig([\tau,\infty))\} 
\mathrm{~for~}A\subseteq X.
\end{equation}
By the assumptions on $\EE_\sig$ for each $\sig\in\Sig$, it is found that  
\begin{equation*}
R_\sig(t,\tau)A\subseteq R_\sig(t,s)R_\sig(s,\tau)A
\end{equation*}
for each $A\subseteq X$, $t\ge s\ge\tau\in\R$. Using the following Lemma, 
one can reduce a nonautonomous evolutionary system to an evolutionary system.

\begin{lem}\cite{CL12}
\label{trajshift}
Let $\tau_0\in\R$ be fixed. Then, for any $\tau\in\R$ and $\sig\in\Sig$, 
there exists at least one $\sig'\in\Sig$ so that 
\begin{equation*}
\EE_\sig([\tau,\infty))=\{u(\cdot):u(\cdot+\tau-\tau_0)\in\EE_{\sig'}([\tau,
\infty))\}.
\end{equation*}
\end{lem}

Thus, it is found that for $A\subseteq X$, $\tau\in\R$ and $t\ge0$,
\begin{equation*}
\bigcup_{\sig\in\Sig}R_\sig(t,0)A=\bigcup_{\sig\in\Sig}R_\sig(t+\tau,\tau)A.
\end{equation*}
Moreover, defining 
\begin{equation*}
\EE_\Sig(I):=\bigcup_{\sig\in\Sig}\EE_\sig(I)
\end{equation*}
for $I\in\TT$, then $\EE_\Sig$ defines an autonomous evolutionary system. 
Also, for $A\subseteq X$, $t\ge 0$, the map $R_\Sig(t):\PP(X)\rightarrow
\PP(X)$ is given by
\begin{equation*}
R_\Sig(t)A:=\bigcup_{\sig\in\Sig}R_\sig(t,0)A.
\end{equation*}
Let $\wwd^\Sig(A)$ be the corresponding omega-limit set for $\EE_\Sig$. That 
is, 
\begin{equation*}
\wwd^\Sig(A):=\bigcap_{T\ge0}\overline{\bigcup_{t\ge T} R_\Sig(t)}^\bullet 
      =\bigcap_{T\ge0}\overline{\bigcup_{t\ge T}
                     \bigcup_{\sig\in\Sig} R_\sig(t,0)}^\bullet.
\end{equation*}

\begin{define}\cite{CL12}
For the autonomous evolutionary system $\EE_\Sig$, we denote its $\dd$-global 
attractor (if it exists) by $\AAd^\Sig$. We call $\AAd^\Sig$ the $\dd$-uniform 
attractor for $\EE_\Sig$.
\end{define}

The following results for $\EE_\Sig$ are then attained using 
Theorem~\ref{globalattractorresults}.

\begin{thm}\cite{CL12}
\label{unifattractor1}
Let $\EE_\Sig$ be a nonautonomous evolutionary system. Then, if the 
$\dd$-uniform attractor exists, then $\AAd^\Sig=\wwd^\Sig(X)$. Also, the weak 
uniform attractor $\AAw^\Sig$ exists.
\end{thm}

Here are additional assumptions imposed on $\EE_\Sig$.
\begin{enumerate}
\refitem{C1} $\EE_\Sig([0,\infty))$ is precompact in $C([0,\infty);\Xw)$.
\refitem{C2} Assume that $X$ is a set in some Banach space $H$ satisfying the 
Radon-Riesz property with the norm denoted by $\abs{\cdot}$, such that 
$\ds(x,y):=\abs{x-y}$ for all $x,y\in X$ and $\dw$ induces the weak topology on 
$X$. Assume also that for any $\ee>0$, there exists $\delta:=\delta(\ee)$, so 
that for each $u\in\EE_\Sig([0,\infty))$ and $t>0$, 
\begin{equation*}
\abs{u(t)}\le\abs{u(t_0)}+\ee,
\end{equation*}
for $t_0$ a.e. in $(t-\delta,t)$.
\refitem{C3} Let $u_k\in\EE_\Sig([0,\infty))$ be so that $u_k$ is 
$d_{C([0,T];\Xw)}$-Cauchy sequence in $C([0,T];\Xw)$ for some $T>0$. Then, 
$u_k(t)$ is $\ds$-Cauchy for a.e. $t\in[0,T]$.
\end{enumerate}
Next, the closure of the evolutionary system $\EE_\Sig$ is introduced. This is 
given by $\bar{\EE}$ defined as follows:
\begin{equation*}
\bar{\EE}([\tau,\infty))
   :=\overline{\EE_\Sig([\tau,\infty))}^{C([\tau,\infty);\Xw)}
\end{equation*}
for each $\tau\in\R$. This is an evolutionary system. Let $\bar{\wwd}(A)$ and 
$\bar{\AAd}$ be the corresponding omega-limit set and global attractor for 
$\bar{\EE}$, respectively. Then, $\bar{\EE}$ has the following properties:

\begin{lem}\cite{CL12}
\label{closureprops}
If $\EE_\Sig$ satisfies $\iref{C1}$, then $\bar{\EE}$ satisfies $\iref{B1}$. 
Moreover, if $\EE_\Sig$ satisfies $\iref{C2}$ and $\iref{C3}$, then 
$\bar{\EE}$ satisfies $\iref{B2}$ and $\iref{B3}$.
\end{lem}

\begin{thm}\cite{CL12}
\label{unifattractor2}
Assume that $\EE_\Sig$ satisfies $\iref{C1}$. Then, the weak uniform attractor 
exists by Theorem~\ref{unifattractor1}). Also, 
\begin{enumerate}
\item $\AAw^\Sig=\www^\Sig(X)=\bar{\www}(X)=\bar{\wws}(X)=\bar{\AAw}=\{u_0\in 
X: u_0=u(0)\mathrm{~for~some~}u\in\bar{\EE}((-\infty,\infty))\}$.
\item $\AAw^\Sig$ is the maximal invariant and maximal quasi-invariant set 
with respect to $\bar{\EE}$.
\item (Weak uniform tracking property) For any $\ee>0$, there exists a $t_0:=
t_0(\ee,T)$ so that for any $t^*>t_0$, every trajectory $u\in\EE_\Sig([0,
\infty))$ satisfies $d_{C([t^*,\infty);\Xw)}(u,v)<\ee$ for some complete 
trajectory $v\in\bar{\EE}((-\infty,\infty))$.
\end{enumerate}
If $\EE_\Sig$ is an asymptotically compact evolutionary system (not necessarily 
satisfying $\iref{C1}$), then 
\begin{enumerate}
\addtocounter{enumi}{3}
\item The strong uniform attractor $\AAs^\Sig$ exists, is strongly compact, 
and $\AAs^\Sig=\AAw^\Sig$.
\end{enumerate}
Furthermore, if $\EE_\Sig$ is asymptotically compact and satisfies $\iref{C1}$, 
then
\begin{enumerate}
\addtocounter{enumi}{4}
\item (Strong uniform tracking property) For any $\ee>0$ and $T>0$, there 
exists $t_0:=t_0(\ee,T)$ so that for $t^*>t_0$, every trajectory $u\in
\EE_\Sig([0,\infty))$ satisfies $\ds(u(t),v(t))<\ee$ for each $t\in
[t^*,t^*+T]$, for some complete trajectory $v\in\bar{\EE}((-\infty,\infty))$.
\end{enumerate}
\end{thm}

\begin{thm}\cite{CL12}
\label{123NES}
Let $\EE_\Sig$ be an evolutionary system satisfying $\iref{C1}$, $\iref{C2}$, 
and $\iref{C3}$, and assume that $\bar{\EE}((-\infty,\infty))\subseteq 
C((-\infty,\infty);\Xs)$. Then, $\EE_\Sig$ is asymptotically compact.
\end{thm}

Let $\EE$ be a nonautonomous evolutionary system with symbol space $\Sig$ and 
shift operators $T(s):\Sig\rightarrow\Sig$ for each $s\ge0$. Then, we have 
that $P_\sig(t,s)=R_\sig(t,s)$ for all $t\ge s$. We can use property (2) in 
Definition~\ref{nonautevolutionarysystem}, to obtain the following identity 
for any $\sig\in\Sig$, $t\ge r\in\R$, $s>0$ and $A\subseteq X$, 
\begin{equation}
\label{shiftproperty}
R_\sig(t+s,r+s)A=R_{T(s)\sig}(t,r)A.
\end{equation}
Using this fact, we have can say that 
\begin{equation*}
\WWd^\sig(A,t+s)=\WWd^{T(s)\sig}(A,t)
\end{equation*}
for any $\sig\in\Sig$, $t\in\R$, and $A\subseteq X$. Thus, we get that 
\begin{equation*}
\bigcup_{\sig\in\Sig}\bigcup_{t\in\R}\WWd^\sig(A,t)
    =\bigcup_{\sig\in\Sig}\WWd^\sig(A,t_0).
\end{equation*}
for any fixed $t_0\in\R$. Now, we have can draw the following relationship 
between $\cup_\sig\WWd^\sig(A,t_0)$ and the uniform attractor $\AAd^\Sig$ (if 
it exists).

\begin{thm}
\label{unifsections}
Let $\EE_\sig$ be a nonautonomous evolutionary system. Then, if the 
$\dd$-uniform attractor $\AAd^\Sig$ exists, we have that 
\begin{equation*}
\overline{\bigcup_{\sig\in\Sig}\WWd^\sig(X,t_0)}^\bullet
  \subseteq\AAd^\Sig
\end{equation*}
for any fixed $t_0\in\R$.
\end{thm}

\begin{rmk}
\label{sequenceversions2}
Note that $x\in\wwd^\Sig(A)$ if and only if there exist sequences 
$\sig_n\in\Sig$, $t_n\ge0$ with $t_n\rightarrow\infty$, and $x_n\in 
R_{\sig_n}(t_n,0)A$ with $x_n\xrightarrow{\dd}x$. Similarly, if $x\in
\overline{\cup_\sig\WWd^\sig(A,t_0)}^\bullet$ then there exist sequences 
$\sig_n\in\Sig$, $s_n\in\R$, $s_n\le t_0$ with $s_n\rightarrow-\infty$, and 
$x_n\in R_{\sig_n}(t_0,s_n)A$ so that $x_n\xrightarrow{\dd}x$. 
\end{rmk}

\begin{proof}
Let $x\in\cup_\sig\WWd^\sig(X,t_0)$. Then, there exist sequences 
$\sig_n\in\Sig$, $s_n\rightarrow-\infty$ with $s_n\le t_0$ and $x_n\in 
R_{\sig_n}(t_0,s_n)X$ with $x_n\xrightarrow{\dd}x$. Without loss of generality, 
we can pass to a subsequence and assume that $s_n\le 0$ for each $n$. Using the 
fact that $R_{\sig_n}(t_0,s_n)X=R_{\sig_n'}(t-s_n,0)X$ for any $\sig_n'$ so 
that $T(-s_n)\sig_n'=\sig_n$, we see that $x\in\wwd^\Sig(X)$ by 
Remark~\ref{sequenceversions2}. Thus, 
\begin{equation*}
\overline{\bigcup_{\sig\in\Sig}\WWd^\sig(X,t_0)}^\bullet
     \subseteq\wwd^\Sig(X)=\AAd^\Sig
\end{equation*}
by Theorem~\ref{unifattractor1}.
\end{proof}

Using this result, Theorem~\ref{unifattractor2}, Theorem~\ref{existsiff}, and 
Theorem~\ref{weakattractorexists}, we get the following corollary.

\begin{cor}
\label{generalweakunifpullback}
Let $\EE_\sig$ be a nonautonomous evolutionary system. Then, the weak uniform 
attractor $\AAw^\Sig$ exists. Similarly, for each $\sig\in\Sig$, the induced 
generalized evolutionary system, there exists a weak pullback attractor 
$\AAw^\sig(t)$. Moreover, 
\begin{equation*}
\overline{\bigcup_{\sig\in\Sig}\AAw^\sig(t_0)}^w \subseteq\AAw^\Sig
\end{equation*}
for any fixed $t_0\in\R$.
\end{cor}

Combining Theorem~\ref{unifsections} with Theorem~\ref{123NES}, the second half 
of Theorem~\ref{unifattractor2}, Theorem~\ref{strongexists}, and the following 
Lemma~\ref{NACimpliesPAC}, we get get a similar embedding of the union of the 
strong pullback attractors within the strong uniform attractor. But first, we need to 
know that the asymptotic compactness of $\EE_\Sig$ guarantees the pullback 
asymptotic compactness of each $\EE_\sig$. This is given in the following lemma.

\begin{lem}
\label{NACimpliesPAC}
Let $\EE_\Sig$, the induced autonomous evolutionary system from the 
nonautonomous evolutionary system $\EE_\sig$ be asymptotically compact. Then, 
for each fixed $\sig\in\Sig$, the induced generalized evolutionary system 
$\EE_\sig$ is pullback asymptotically compact.
\end{lem}

\begin{proof}
Let $\sig\in\Sig$ be fixed. Let $s_n\le t$ be so that $s_n\rightarrow-\infty$. 
Also, let $x_n\in R_\sig(t,s_n)X$. Then, using (\ref{shiftproperty}), we get 
that 
\begin{equation*}
x_n\in\bigcup_{\sig\in\Sig}R_\sig(t,s_n)X=\bigcup_{\sig\in\Sig}R_\sig(t-s_n,0)X
            =R_\Sig(t-s_n)X
\end{equation*}
for $t-s_n\rightarrow\infty$. Thus, $x_n$ has a convergent subsequence by the 
asymptotic compactness of $\EE_\Sig$
\end{proof}

\begin{thm}
\label{unifsections2}
Let $\EE_\Sig$, the induced autonomous evolutionary system from the 
nonautonomous evolutionary system $\EE_\sig$, be asymptotically compact or let 
$\EE_\Sig$ satisfy $\iref{C1}$, $\iref{C2}$, and $\iref{C3}$ with complete 
trajectories strongly continuous. Then, the weak uniform attractor $\AAw^\Sig$ 
is the strongly compact strong uniform attractor $\AAs^\Sig$. Also, for each 
fixed $\sig\in\Sig$, the weak pullback attractor $\AAw^\sig(t)$ is a strongly 
compact strong pullback attractor $\AAs^\sig(t)$. Moreover, 
\begin{equation*}
\overline{\bigcup_{\sig\in\Sig}\AAw^\sig(t_0)}^w
  =\overline{\bigcup_{\sig\in\Sig}\AAs^\sig(t_0)}^s
  \subseteq\AAs^\Sig=\AAw^\Sig
\end{equation*}
for any fixed $t_0\in\R$.
\end{thm}

Unfortunately, the reverse inclusion is not known in full generality at this 
time. However, with the inclusion of some fairly weak assumptions, we can prove 
the reverse inclusion. To this end, let $\EE_\sig$ be a nonautonomous 
evolutionary system. Assume that $\EE_\Sig$, the induced evolutionary system 
from $\EE_\sig$ satisfies $\iref{C1}$. Assume that for a fixed $\sig\in\Sig$ 
that $\EE_\sig([s,\infty))$ is closed in $C([s,\infty);\Xw)$ for each $s\in\R$. 
Then, $\EE_\sig$ satisfies $\iref{A1}$. Let us denote by $\AAd^\sig(t)$ the 
$\dd$-pullback attractor for the generalized evolutionary system $\EE_\sig$ 
(if it exists). Using Theorem~\ref{unifattractor2} and Theorem~\ref{weak=i}, 
we get the following.

\begin{thm}
\label{unif=pullbacksections}
Let $\EE_\sig$ be a nonautonomous evolutionary system. Let $\EE_\Sig$, the 
induced evolutionary system satisfy $\iref{C1}$. Assume that for any $\sig\in
\Sig$ and any $s\in\R$ that $\EE_\sig([s,\infty))$ is closed in 
$C([s,\infty);\Xw)$. Then the weak uniform attractor $\AAw^\Sig$ is given by 
\begin{equation*}
\AAw^\Sig=\overline{\bigcup_{\sig\in\Sig}\AAw^\sig(0)}^w
\end{equation*}
where $\AAw^\sig(t)$ is the weak pullback attractor for the generalized 
evolutionary system $\EE_\sig$. Moreover, the weak uniform tracking property 
holds.
\end{thm}

\begin{proof}
For the closure of $\EE_\Sig$, $\bar{\EE}$, we have using 
Theorem~\ref{unifattractor2} that the weak uniform attractor $\AAw^\Sig$ 
is given by 
\begin{equation*}
\AAw^\Sig = \{u(0):u\in\bar{\EE}((-\infty,\infty))\}.
\end{equation*}
Also, by Theorem~\ref{weak=i}, we have that for each $\sig\in\Sig$, 
\begin{equation*}
\AAw^\sig(t)=\II_\sig(t):=\{u(t):u\in\EE_\sig((-\infty,\infty))\}.
\end{equation*}
Thus, if we can show that 
\begin{equation*}
\{u(0):u\in\bar{\EE}((-\infty,\infty))\}
  = \overline{\bigcup_{\sig\in\Sig}\{u(0):u\in\EE_\sig((-\infty,\infty))\}}^w,
\end{equation*}
we are done. First, let $u\in\bar{\EE}((-\infty,\infty))$. Then, there are 
$u_n\in \EE_{\sig_n}((-\infty,\infty))$ for some $\sig_n\in\Sig$ with 
$u_n\rightarrow u$ in the sense of $C((-\infty,\infty);\Xw)$. Then, $u_n(0)
\xrightarrow{\dw}u(0)$. Therefore, 
\begin{equation*}
\{u(0):u\in\bar{\EE}((-\infty,\infty))\}\subseteq
\overline{\bigcup_{\sig\in\Sig}\{u(0):u\in\EE_\sig((-\infty,\infty))\}}^w.
\end{equation*}
For the other inclusion, let $u_n\in\EE_{\sig_n}((-\infty,\infty))$ be so that 
$u_n(0)\xrightarrow{\dw}x$ for some $x\in X$. Since $\EE_\Sig$ satisfies 
$\iref{C1}$, there is a subsequence which we reindex as $u_n$ converging in 
the sense of $C([0,\infty);\Xw)$ to some $u^0\in\bar{\EE}([0,\infty))$. In 
particular, $u_n(0)\xrightarrow{\dw}x$. Hence, $u^0(0)=x$. Again, passing to 
another subsequence, we can find a subsequence and drop a subindex to obtain 
that $u_n\rightarrow u^1\in\bar{\EE}([-1,\infty))$ in $C([-1,\infty);\Xw)$. 
Note that then $u^1|_{[0,\infty)}=u^0$. By the usual diagonalization argument, 
we find a subsequence $u_n\rightarrow u$ for some $u\in\bar{\EE}((-\infty,
\infty))$ in $C((-\infty,\infty);\Xw)$ with $u(0)=x$. Therefore, 
\begin{equation*}
\{u(0):u\in\bar{\EE}((-\infty,\infty))\}\supseteq
  \overline{\bigcup_{\sig\in\Sig}\{u(0):u\in\EE_\sig((-\infty,\infty))\}}^w.
\end{equation*}
\end{proof}

Finally, we combine this with Lemma~\ref{NACimpliesPAC}, Theorem~\ref{123NES}, 
Theorem~\ref{unifattractor2}, and Theorem~\ref{strongexists} to get the 
following Theorem.

\begin{thm}
\label{strongsections}
Let $\EE_\sig$ be a nonautonomous evolutionary system. Let $\EE_\Sig$, the 
induced evolutionary system satisfy $\iref{C1}$. Assume that, for any $\sig\in
\Sig$ and any $s\in\R$ that $\EE_\sig([s,\infty))$ is closed in $C([s,\infty;
\Xw)$. Moreover, assume that $\EE_\Sig$ is asymptotically compact or that 
$\EE_\Sig$ satisfies $\iref{C2}$ and $\iref{C3}$ with complete trajectories 
strongly continuous. Then, we have that 
\begin{equation*}
\AAs^\Sig=\AAw^\Sig=\overline{\bigcup_{\sig\in\Sig}\AAw^\sig(0)}^w
                   =\overline{\bigcup_{\sig\in\Sig}\AAs^\sig(0)}^w.
\end{equation*}
\end{thm}

\section{Application to the 3D Navier-Stokes Equations}
\label{3dnses}
The 3D space-periodic, incompressible Navier-Stokes Equations (NSEs) on the 
periodic domain $\Omega:=\mathbb{T}^3$, the three-dimensional torus, are 
given as follows:
\begin{equation}
\label{NSE}
\left\{\begin{array}{l}\derivt u-\nu\Delta u+(u\cdot\nabla)u+\nabla 
            p = f(t) \\ 
       \nabla\cdot u = 0\end{array}\right.
\end{equation}
where $u$, the velocity and $p$, the pressure, are unknowns; $f(t)$ is a 
time-dependent forcing term; and $\nu>0$ is the kinematic viscosity 
coefficient of the fluid. Assuming the initial condition, $u_s:=u(s)$, and the 
forcing term, $f(t)$ for each $t$, have the property that 
\begin{equation*}
\int_\Omega u_s(x)\ddx=\int_\Omega f(x,t)\ddx=0,
\end{equation*}
we get that 
\begin{equation*}
\int_\Omega u(x,t)\ddx=0
\end{equation*}
for all $t\ge s\in\R$. We are now ready to set up the functional setting.

Denote by $\ip{\cdot}{\cdot}$ and $\abs{\cdot}$ the $L^2(\Omega)^3$-inner 
product and the $L^2(\Omega)^3$-norm, respectively. Let $\mathscr{V}$ be 
given by
\begin{equation*}
\mathscr{V}:=\left\{u\in [C^\infty(\Omega)]^3:\int_\Omega u(x)\ddx
                       =0,~\nabla\cdot u=0\right\}.
\end{equation*}
Next, let $H$ and $V$ be the closures of $\mathscr{V}$ in $L^2(\Omega)^3$ 
and $H^1(\Omega)^3$, respectively. Denote by $H_w$ the set $H$ endowed with 
the weak topology.

Let $P_\sigma:L^2(\Omega)^3\rightarrow H$ be the $L^2$ orthogonal projection, 
known as the Leray projector. Let $A:=-P_\sigma\Delta=-\Delta$ be the Stokes 
operator with domain $(H^2(\Omega))^3\cap V$. Note that the Stokes operator is 
a self-adjoint, positive operator with compact inverse. Let 
\begin{equation*}
\norm{u}:=\abs{A^{1/2}u}.
\end{equation*}
Note that $\norm{u}$ is equivalent to the $H^1$ norm of $u$ for $u\in 
D(A^{1/2})$ by the Poincar\'{e} inequality. Let $\IP{\cdot}{\cdot}$ denote the 
corresponding inner product in $H^1$.

Next, denote by $B(u,v):=P_\sigma(u\cdot\nabla v)\in V'$ for each $u,v\in V$. 
This is a bilinear form with the following property:
\begin{equation*}
\angip{B(u,v)}{w}=-\angip{B(u,w)}{v}
\end{equation*}
for each $u,v,w\in V$.

We can now rewrite (\ref{NSE}) as a differential equation in $V'$. That is, 
\begin{equation}
\label{NSEweak}
\derivt u+\nu Au+B(u,u)=g
\end{equation}
for $g:=P_\sigma f$, and $u$ is a $V$-valued function of time.

\begin{define}
\label{weaksoln}
The function $u:[T,\infty)\rightarrow H$ (or $u:(-\infty,\infty)\rightarrow 
H$) is a weak solution to (\ref{NSE}) on $[T,\infty)$ (or $(-\infty,\infty)$) 
if
\begin{enumerate}
\item $\derivt u\in L_{loc}^1([T,\infty);V')$.
\item $u\in C([T,\infty);H_w)\cap L^2_{loc}([T,\infty);V)$.
\item $\ip{\derivt u(t)}{\phi}+\nu\IP{u(t)}{\phi}+\angip{B(u(t),u(t))}{\phi}
           =\angip{g(t)}{\phi}$ for a.e. $t\in[T,\infty)$ and each $\phi\in V$.
\end{enumerate}
\end{define}

\begin{thm}[Leray, Hopf]
\label{LHsolns}
For each $u_0\in H$ and $g\in L^2_{loc}(\R;V')$, there exists a weak solution 
of (\ref{NSE}) on $[T,\infty)$ with $u(T)=u_0$, and for each $t\ge t_0$, 
$t_0$ a.e. in $[T,\infty)$ we have the following energy inequality:
\begin{equation}
\label{LHenergyineq}
\abs{u(t)}^2+2\nu\int_{t_0}^t\norm{u(s)}^2\dds\le\abs{u(t_0)}^2+2\int_{t_0}^t
   \angip{g(s)}{u(s)}\dds.
\end{equation}
\end{thm}

\begin{define}
\label{LHweaksoln}
A weak solution to (\ref{NSE}) satisfying (\ref{LHenergyineq}) will be called 
a Leray-Hopf weak solution.
\end{define}

\begin{define}
\label{leraysoln}
A weak solution to (\ref{NSE}) on $[T,\infty)$ satisfying the energy 
inequality 
\begin{equation*}
\abs{u(t)}^2+2\nu\int_T^t\norm{u(s)}^2\dds\le\abs{u(T)}^2+2\int_T^t\angip{g(s)}
                                                            {u(s)}\dds
\end{equation*}
for each $t\ge T$ is called a Leray solution.
\end{define}

Leray solutions are special in that they are continuous at the starting time 
$T$. Note that via the Galerkin method, one can prove the existence of Leray 
solutions for each $u_0\in H$ and $T\in\R$. 

Assume $g$ is translationally bounded in $L^2_{\mathrm{loc}}(\R,V')$. That is, 
\begin{equation*}
\norm{g}_{L^2_b}^2:=\sup_{t\in\R}\int_t^{t+1}\norm{g(s)}^2_{V'}\dds<\infty.
\end{equation*}
We will show that there exists a bounded set $X\subset H$ which captures all 
of the asymptotic dynamics of Leray solutions with a translationally bounded 
force $g$. We will be more precise with what this means in a moment. But 
first, we need a preliminary definition and a lemma. The lemma's proof can be 
found in \cite{CV02}.

\begin{define}
A function $f(s)$ is almost everywhere equal to a monotonic non-increasing 
function on $[a,b]$ if $f(t)\le f(\tau)$ for any $t,\tau\in[a,b]\backslash Q$ 
with $\tau\le t$ and the measure of $Q$ is zero.
\end{define}

\begin{lem}
\label{CVgronwall}
Let $f(s)\in L_1([a,b])$. Then, the function $f(s)$ is almost everywhere 
equal to a monotone non-increasing function on $[a,b]$ if and only if, for 
any $\phi\in C_0^\infty((a,b))$ with $\phi(s)\ge0$, one has 
\begin{equation*}
\int_a^b f(s)\phi'(s)\dds\ge0.
\end{equation*}
\end{lem}

So, let $u$ be a Leray solution to (\ref{NSE}) with $g$ translationally 
bounded and starting time $t_0$. That is, a point where the energy inequality 
(\ref{LHenergyineq}) is satisfied. Then, applying Young's inequality followed 
by the Poincar\`{e} inequality, we find that 
\begin{equation}
\label{energyineqform2}
\abs{u(t)}^2+\nu\lambda_1\int_{t_0}^t\abs{u(s)}^2\dds\le\abs{u(t_0)}^2+
    \frac{1}{\nu}\int_{t_0}^t\norm{g(s)}^2_{V'}\dds.
\end{equation}
Let $\phi\in C_0^\infty((t_0,\tau))$ for $\tau\ge t$. Then, the above 
inequality is equivalent to the following distributional inequality
\begin{equation}
\label{LHdistform}
-\int_{t_0}^\tau\abs{u(s)}^2\phi'(s)\dds+\nu\lambda_1\int_{t_0}^\tau
    \abs{u(s)}^2\phi(s)\dds\le\frac{1}{\nu}\int_{t_0}^\tau
       \norm{g(s)}^2_{V'}\phi(s)\dds.
\end{equation}
Replacing $\phi$ with $e^{\nu\lambda_1 s}\phi\in C_0^\infty((t_0,\tau))$, 
we have that 
\begin{align*}
-\int_{t_0}^\tau\abs{u(s)}^2e^{\nu\lambda_1 s}\phi'(s)\dds
   &\le\frac{1}{\nu}\int_{t_0}^\tau\norm{g(s)}_{V'}^2
         e^{\nu\lambda_1 s}\phi(s)\dds \\
   &=\frac{1}{\nu}\int_{t_0}^\tau\frac{\mathrm{d}}{\mathrm{ds}}
         \left(\int_0^s \norm{g(r)}_{V'}^2e^{\nu\lambda_1 r}\ddr\right)
           \phi(s)\dds \\
   &=-\frac{1}{\nu}\int_{t_0}^\tau\left(\int_0^s\norm{g(r)}_{V'}^2e^{\nu
         \lambda_1 r}\ddr\right)\phi'(s)\dds.
\end{align*}
Rearranging, and using Lemma~\ref{CVgronwall}, we get that 
\begin{equation}
\label{almostthere}
\abs{u(t)}^2e^{\nu\lambda_1 t}-\abs{u(t_0)}^2e^{\nu\lambda_1 t_0} 
  \le\frac{1}{\nu}\int_{t_0}^t\norm{g(s)}^2_{V'}e^{\nu\lambda_1 s}\dds.
\end{equation}
It remains to estimate the right-hand side of (\ref{almostthere}). As in 
\cite{LWZ05}, we have that 
\begin{align*}
\int_{t_0}^t\norm{g(s)}^2_{V'}e^{\nu\lambda_1 s}\dds 
  &\le\int_{t-1}^t\norm{g(s)}_{V'}^2e^{\nu\lambda_1 s}\dds +
    \int_{t-2}^{t-1}\norm{g(s)}_{V'}^2e^{\nu\lambda_1 s}\dds +\cdots \\
  &\le e^{\nu\lambda_1t}\left(1+e^{-\nu\lambda_1}+e^{-2\nu\lambda_1}
    \right)\sup_{t\in\R}\int_t^{t+1}\norm{g(s)}_{V'}^2\dds \\
  &\le\frac{e^{\nu\lambda_1 t}}{1-e^{-\nu\lambda_1}}\norm{g}_{L^2_b}^2.
\end{align*}
Thus, we arrive at the following inequality
\begin{equation*}
\label{absorbing}
\abs{u(t)}^2\le\abs{u(t_0)}^2e^{\nu\lambda_1(t_0-t)}+
   \frac{1}{\nu}\frac{\norm{g}_{L^2_b}^2}{1-e^{-\nu\lambda_1}}.
\end{equation*}

Let 
\begin{equation*}
R:=\frac{2\norm{g}_{L^2_b}^2}{\nu(1-e^{-\nu\lambda_1})}.
\end{equation*}
Then, $X$ given by 
\begin{equation*}
X:=\{u\in H:\abs{u}\le R\}
\end{equation*}
is a closed absorbing ball in $H$ for Leray solutions. Moreover, $X$ is weakly 
compact and contains all of the asymptotic dynamics of Leray solutions by the 
above argument. Define the strong and weak distances on $X$, respectively, by 
\begin{equation*}
\ds(u,v):=\abs{u-v}\quad\mathrm{~and~}\quad\dw(u,v):=\sum_{k\in\Z^3}
   \frac{1}{2^{\abs{k}}}
       \frac{\abs{\hat{u_k}-\hat{v_k}}}{1+\abs{\hat{u_k}-\hat{v_k}}}
\end{equation*}
for $u,v\in H$ where $\hat{u_k}$ and $\hat{v_k}$ are the Fourier coefficients 
of $u$ and $v$, respectively. Note that the above weak metric $\dw$ induces 
the weak topology $H_w$ on $X$. Next, we define our generalized evolution 
system on $X$ by 
\begin{align*}
\EE([T,\infty)):=&\{u:u\mathrm{~is~a~Leray{-}Hopf~solution~of~(\ref{NSE})~
            on~}[T,\infty) \\
                 &\mathrm{~and~}u(t)\in X\mathrm{~for~}t\in[T,\infty)\}, \\
\EE((-\infty,\infty)):=&\{u:u\mathrm{~is~a~Leray{-}Hopf~solution~of~(\ref{NSE})
       ~on~}(-\infty,\infty) \\
                 &\mathrm{~and~}u(t)\in X\mathrm{~for~}t\in(-\infty,\infty)\}.
\end{align*}
Then, $\EE$ satisfies the necessary properties in Definition~\ref{GES} and 
forms a generalized evolutionary system on $X$. We must use Leray-Hopf 
solutions as our generalized evolutionary system since the restriction of a 
Leray solution may not be a Leray solution, but it is always a Leray-Hopf 
solution.

Note that an absorbing ball does not exist for the Leray-Hopf weak solutions. 
An absorbing ball is a bounded set $X\subset H$ so that for any $B\subset H$ 
bounded and any $s_0\in\R$, there is some $\sig:=\sig(B)\ge s_0$ so that if 
$u(s_0)\in B$, then $u(s)\in X$ for $s\ge \sig$. This requires uniformity in 
$B$. However, Leray-Hopf solutions may have ``jumps" at the starting point 
which can be as large as you like. Thus, even if you were to consider the 
bounded set $B:=\{0\}$ for the Leray-Hopf solutions with $s_0:=0$ in the 
autonomous case, you may not have such a structure as Figure~\ref{lerayfig} 
illustrates. 

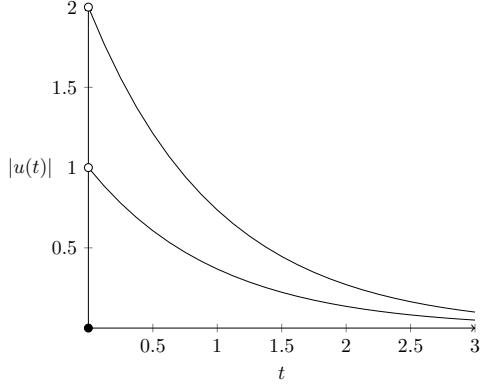
\begin{figure}
\begin{center}
\begin{tikzpicture}[scale=.75]
\begin{axis}
  [
    axis lines=middle, 
    axis line style={->},
    x label style={at={(axis description cs:0.5,-0.1)},anchor=north},
    y label style={at={(axis description cs:-0.15,0.55)},anchor=north},
    xlabel={$t$},
    ylabel={$\abs{u(t)}$}
  ]  
  \addplot[domain=0:3,black] {exp(-x)};
  \addplot[domain=0:3,black] {2*exp(-x)};
  \addplot[holdot] coordinates{(0,1)(0,2)};
  \addplot[soldot] coordinates{(0,0)};
\end{axis}
\end{tikzpicture}
\end{center}
\caption{Possible Leray-Hopf weak solutions for the Navier-Stokes equations 
with zero forcing.}
\label{lerayfig}
\end{figure}

By Theorem~\ref{exists}, $\EE$ has a weak pullback attractor. Next, we will 
show that $\EE$ satisfies $\iref{A1}$ and $\iref{A3}$ under the current 
conditions, and $\EE$ satisfies $\iref{A2}$ under an additional assumption 
which we will name later. We start with a preliminary lemma.

\begin{lem}
\label{NSEconvprops}
Let $u_n$ be a sequence of Leray-Hopf weak solutions of (\ref{NSE}) on 
$[s,\infty)$, so that $u_n(t)\in X$ for all $t\ge s$ for some $s\in\R$. Then, 
there exists a subsequence $n_j$ so that $u_{n_j}$ converges to some $u$ in 
$C([t_1,t_2];H_w)$. That is, 
\begin{equation*}
(u_{n_j},v)\rightarrow(u,v)
\end{equation*}
uniformly on $[t_1,t_2]$ as $n_j\rightarrow\infty$ for all $v\in H$.
\end{lem}

\begin{proof}
The major arguments in this lemma are classical. For more information, see 
\cite{CF88}, \cite{T84}, and \cite{R01}, among others. 

First, using (\ref{LHenergyineq}) as well as the definition of $\EE$, we have 
that $u_n$ is uniformly bounded in $L^\infty([t_1,t_2];H)$ and in 
$L^2((t_1,t_2);V)$. Thus, we use Alaoglu compactness theorem to find 
subsequences (which we will keep reindexing as $u_n$) which converge to 
some $u$ weak-* in $L^\infty((t_1,t_2);H)$ and weakly in 
$L^2((t_1,t_2);V)$.

Next, using the fact that $A:V\rightarrow V'$ is continuous using the 
assignment 
\begin{equation*}
\angip{Av}{\phi}:=\ip{A^{1/2}v}{A^{1/2}\phi}=\IP{v}{\phi},
\end{equation*}
we get that $Au_n$ is uniformly bounded in $L^2((t_1,t_2);V')$. Thus, we 
can extract a subsequence and relable as $u_n$ so that $Au_n$ converges to 
$Au$ weakly in $L^2((t_1,t_2);V')$.

A classical estimate gives us that 
\begin{equation*}
\norm{B(u,u)}_{V'}\le C\abs{u}^{1/2}\norm{u}^{3/2}
\end{equation*}
for any $u\in V$ and $C$ a constant. Thus, using the fact that $u_n$ is 
uniformly bounded in both $L^{\infty}((t_1,t_2);H)$ and $L^2((t_1,t_2);V)$, 
we see that $B(u_n,u_n)$ is uniformly bounded in $L^{4/3}((t_1,t_2);V')$ 
using the following estimates:
\begin{align*}
\norm{B(u_n,u_n)}_{L^{4/3}((t_1,t_2);V')}^{4/3} 
  &\le C\int_{t_1}^{t_2}\abs{u_n(s)}^{2/3}\norm{u_n(s)}^2\dds \\
  &\le C\norm{u_n}^{2/3}_{L^\infty((t_1,t_2);H)}
        \norm{u_n}^2_{L^2((t_1,t_2);V)}.
\end{align*}
Since $Au_n$, $B(u_n,u_n)$, and $g$ are uniformly bounded sequences in 
$L^{4/3}((t_1,t_2);V')$, we use (\ref{NSEweak}) to say that so is 
$\derivt u_n$. So, we have that $B(u_n,u_n)$ converges weakly in 
$L^{4/3}((t_1,t_2);V')$ and $\derivt u_n$ converges weakly in 
$L^{4/3}((t_1,t_2);V')$. Moreover, a standard compactness argument gives 
us that then $u_n$ converges strongly to $u$ in $L^2((t_1,t_2);H)$. Using 
the strong convergence of $u_n$ as well as the uniform bound of $u_n$ in 
$L^{\infty}((t_1,t_2);H)$, a well known result shows that $B(u_n,u_n)$ 
converges weakly to $B(u,u)$ in $L^{4/3}((t_1,t_2);V')$. 

Passing to the limit gives us that 
\begin{equation*}
\derivt u+\nu Au+B(u,u)=g
\end{equation*}
in $V'$. Now, take the inner product with $v\in V$ and integrate from 
$t$ to $t+h$ where $t_1\le t<t+h\le t_2$. We get that 
\begin{equation*}
\ip{u(t+h)-u(t)}{v}=-\nu\int_t^{t+h}\IP{u(r)}{v}\ddr-
                      \int_t^{t+h}\angip{B(u(r),u(r))}{v}\ddr+
                      \int_t^{t+h}\angip{f(r)}{v}\ddr.
\end{equation*}
Taking the absolute value and using Cauchy-Swartz followed by 
H\"{o}lder's inequality, we find that 
\begin{align*}
\abs{\ip{u(t+h)-u(t)}{v}}\le &
    ~\nu\left(\int_t^{t+h}\norm{u(r)}^2\ddr\right)^{1/2}
       \left(\int_t^{t+h}\norm{v}^2\ddr\right)^{1/2}  \\
  &+ \left(\int_t^{t+h}\norm{B(u(r),u(r))}_{V'}^{4/3}\ddr\right)^{3/4}
     \left(\int_t^{t+h}\norm{v}^4\ddr\right)^{1/4} \\
  &+ \left(\int_t^{t+h}\norm{f(r)}_{V'}^2\ddr\right)^{1/2}
     \left(\int_t^{t+h}\norm{v}^2\ddr\right)^{1/2}.
\end{align*}
Using the above uniform bounds then gives us that 
\begin{equation*}
\lim_{h\rightarrow0}\ip{u(t+h)-u(t)}{v}=0.
\end{equation*}
Since $V$ is dense in $H$, $u\in C([t_2,t_1],H_w)$, and we are done.
\end{proof}

\begin{thm}
\label{NSEa1a3}
The generalized evolutionary system $\EE$ satisfies $\iref{A1}$ and 
$\iref{A3}$.
\end{thm}

\begin{proof}
Let $u_n$ be a sequence in $\EE([s,\infty))$ for some $s\in\R$. Then, 
repeatedly using Lemma~\ref{NSEconvprops}, there is a subsequence which 
we reindex as $u_n$ that converges to some $u^1\in C([s,s+1];H_w)$. 
A further subsequence converges to some $u^2\in C([s,s_2];H_w)$ with 
$u^1(t)=u^2(t)$ on $[s,s+1]$. Continuing this diagonalization process, 
we get that there is some subsequence $u_{n_j}$ converging to $u\in 
C([s,\infty),H_w)$. Note that the convergence in 
Lemma~\ref{NSEconvprops} gives us that the energy inequality
\begin{equation*}
\abs{u_n(t)}^2+2\nu\int_{t_0}^t\norm{u_n(s)}^2\dds\le\abs{u_n(t_0)}^2
 +2\int_{t_0}^t\angip{g(s)}{u_n(s)}\dds
\end{equation*}
converges as well to 
\begin{equation*}
\abs{u(t)}^2+2\nu\int_{t_0}^t\norm{u(s)}^2\dds\le\abs{u(t_0)}^2
 +2\int_{t_0}^t\angip{g(s)}{u(s)}\dds
\end{equation*}
for $t\ge t_0$ and $t_0$ a.e. in $[s,\infty)$. That is, $u\in\EE
([s,\infty))$, and $\iref{A1}$ is proven.

For $\iref{A3}$, let $u_n\in\EE[s,\infty)$ for some $s\in\R$ and 
let $u_n\rightarrow u\in\EE([s,\infty))$ in $C([s,t];H_w)$. 
Then, as we saw in Lemma~\ref{NSEconvprops}, $u_n$ is uniformly bounded 
in $L^2([s,t];V)$ for any $T\ge s$ and $\derivt u_n$ is uniformly 
bounded in $L^{4/3}([s,t];V')$ giving us that $u_n\rightarrow u$ 
strongly in $L^2([s,t];H)$. In particular, we have 
\begin{equation*}
\int_s^t\abs{u_n(r)-u(r)}^2\ddr\rightarrow 0
\end{equation*}
as $n\rightarrow\infty$. Thus, $\abs{u_n(t_0)}\rightarrow\abs{u(t_0)}$ 
a.e. on $[s,t]$.
\end{proof}

Therefore, using $\iref{A1}$ we have the following results for $\EE$.

\begin{thm}
\label{a1results}
The weak global attractor for $\EE$, $\AAw(t)$, is the maximal pullback 
quasi-invariant and maximal pullback invariant subset of $X$. Also, 
\begin{equation*}
\AAw(t)=\II(t)=\{u(t):u\in\EE((-\infty,\infty))\}.
\end{equation*}
Moreover, $\EE$ satisfies the weak pullback tracking property and if the strong 
pullback attractor $\AAs(t)$ exists, $\AAw(t)=\AAs(t)$.
\end{thm}

Next, we assume that $g$ is normal in $L^2_{loc}(\R;V')$. That is, the 
following definition introduced in \cite{LWZ05}:
\begin{define}
\label{normal}
Let $Y$ be a Banach space. We say that a function $\phi\in L^2_{loc}(\R;Y)$ is 
normal in $L^2_{loc}(\R;Y)$ if, for any $\ee>0$, there exists a $\delta:=
\delta(\ee)$, so that 
\begin{equation*}
\sup_{t\in\R}\int_t^{t+\delta}\norm{\phi(s)}^2_Y\dds\le\ee.
\end{equation*}
\end{define}
This leads us to the following result.

\begin{thm}
\label{NSEa2}
The generalized evolutionary system $\EE$ with normal forcing term $g$ 
satisfies $\iref{A2}$.
\end{thm}

\begin{proof}
Let $u\in\EE([s,\infty))$ for some $s\in\R$. Let $\ee>0$. Then, using the 
Leray-Hopf energy inequality (\ref{LHenergyineq}), we get that 
\begin{equation*}
\abs{u(t)}^2\le\abs{u(t_0)}^2
   +\frac{1}{\nu}\int_{t_0}^t\norm{g(s)}^2_{V'}\dds
\end{equation*}
from (\ref{energyineqform2}) for all $s\le t_0\le t$, $t_0$ a.e. in 
$[s,\infty)$. Putting this together with the normality of $g$, there is some 
$\delta>0$ so that for $t_0$ a.e. in $(t-\delta,t)$, we have that 
\begin{equation*}
\abs{u(t)}^2\le\abs{u(t_0)}^2+\ee,
\end{equation*}
and $\iref{A2}$ follows.
\end{proof}

Using this, we now have that $\EE$ is pullback asymptotically compact, 
assuming that complete trajectories are strongly continuous. Thus, we can 
deduce the following results for $\EE$.

\begin{thm}
\label{NSEpac}
Suppose the generalized evolutionary system $\EE$ has a normal forcing term 
$g$. Also, suppose that $\EE((-\infty,\infty))\subseteq C((-\infty,\infty);
\Xs)$. Then, $\EE$ has a strongly compact, strong pullback attractor $\AAs(t)$. 
Also, the strong and weak pullback attractors coincide giving us that 
\begin{equation*}
\AAs(t)=\AAw(t)=\II(t)=\{u(t):u\in\EE((-\infty,\infty))\}.
\end{equation*}
That is, $\AAs(t)$ is the maximal pullback invariant and maximal pullback 
quasi-invariant set. Finally, $\EE$ has the strong pullback attracting 
property.
\end{thm}

\begin{proof}
This is a direct consequence of Corollary~\ref{strong=i}, 
Theorem~\ref{strongtracking}, and Theorem~\ref{123implyPAC}.
\end{proof}


\bibliographystyle{plain}
\bibliography{refs}

\end{document}